\documentclass[1 [leqno,11pt]{amsart}
\usepackage{amssymb, amsmath}
\allowdisplaybreaks[4]
\let\pa=\partial
\let\al=\alpha

\let\d=\delta
\let\lam=\lambda

\let\f=\frac
\let\om=\omega

\let\D=\Delta

\let\wt=\widetilde
\let\wh=\widehat
\def\cA{{\mathcal A}}
\def\cB{{\mathcal B}}
\def\cC{{\mathcal C}}
\def\cE{{\mathcal E}}
\def\cF{{\mathcal F}}
\def\cH{{\mathcal H}}
\def\cS{{\mathcal S}}

\def\cD{{\mathcal D}}

\def\h{{\rm h}}

\def\eqdefa{\buildrel\hbox{\footnotesize def}\over =}
\def\K{\mathop{\mathbb K\kern 0pt}\nolimits}
\def\N{\mathop{\mathbb N\kern 0pt}\nolimits}
\def\Q{\mathop{\mathbb Q\kern 0pt}\nolimits}
\def\R{\mathop{\mathbb R\kern 0pt}\nolimits}
\def\Z{\mathop{\mathbb Z\kern 0pt}\nolimits}

\def\dv{\mbox{\rm div}}
\def\dive{\mathop{\rm div}\nolimits}
\def\curl{\mathop{\rm curl}\nolimits}
\def\Supp{\mathop{\rm Supp}\nolimits\ }
\def\no{\noindent}
\def\na{\nabla}
\def\p{\partial}

\newcommand{\andf}{\quad\hbox{and}\quad}

\def\hh{{H^{\frac12,0}}}
\def\hm{{H^{-\frac12,0}}}
\def\hf{{H^{-\frac13,-\frac16}}}
\def\hs{{H^{\frac23,-\frac16}}}

\def\vcurl{v^{\rm h}_{\rm curl}}
\def\Laph{\Delta_{\rm h}}
\def\nablah{\nabla_{\rm h}}
\def\divh{\dive_{\rm h}}
\def\vdiv{v^{\rm h}_{\rm div}}
\def\Ht{\cH_\theta}

\def\vh{v^{\rm h}}

\def\dhk{\Delta_k^{\rm h}}

\def\dvl{\Delta_{\ell}^{\rm v}}

\def\omss{\omega_{\frac34}}
\def\vss{(\pa_3 v^3)_{\f34}}

\def\v{{\rm v}}

\newcommand{\with}{\quad\hbox{with}\quad}
\newcommand{\beq}{\begin{equation}}
\newcommand{\eeq}{\end{equation}}
\newcommand{\ben}{\begin{eqnarray}}
\newcommand{\een}{\end{eqnarray}}
\newcommand{\beno}{\begin{eqnarray*}}
\newcommand{\eeno}{\end{eqnarray*}}
\newtheorem{defi}{Definition}[section]
\newtheorem{thm}{Theorem}[section]
\newtheorem{lem}{Lemma}[section]
\newtheorem{rmk}{Remark}[section]
\newtheorem{cor}{Corollary}[section]
\newtheorem{prop}{Proposition}[section]

\begin{document}
\title[Global well-posedness of 3-D anisotropic Navier-Stokes system]
{Global well-posedness of 3-D anisotropic Navier-Stokes system
with large vertical viscous coefficient}
\author[Y. Liu]{Yanlin Liu}
\address[Y. Liu]{Hua Loo-Keng Center for Mathematical Sciences, Academy of Mathematics and System Sciences,
 The Chinese Academy of Sciences, Beijing 100190, CHINA.} \email{liuyanlin@amss.ac.cn}
\author[P. Zhang]{Ping Zhang} \address[P. Zhang]{Academy of Mathematics $\&$ Systems Science and Hua Loo-Keng Center for Mathematical Sciences, The Chinese Academy of
Sciences, Beijing 100190, CHINA, and School of Mathematical Sciences, University of Chinese Academy of Sciences, Beijing 100049, CHINA.} \email{zp@amss.ac.cn}

\date{\today}

\begin{abstract}
In this paper, we first prove the global well-posedness of
3-D anisotropic  Navier-Stokes system provided that the vertical
viscous coefficient of the system is sufficiently  large compared to some critical norm of the initial data.
 Then we shall construct a family of initial data, $u_{0,\nu},$ which vary fast enough in the vertical variable and which
 are not small in the space, $BMO^{-1}.$ Yet $u_{0,\nu}$ generates a unique global solution to the classical 3-D Navier-Stokes system provided that $\nu$ is sufficiently large.
\end{abstract}
\maketitle



\setcounter{equation}{0}
\section{Introduction}\label{sec1}
In this paper, we first investigate the global well-posedness of
the following 3-D anisotropic  Navier-Stokes system provided that the vertical
viscous coefficient is large enough:
\begin{equation*} (NS_\nu)
\quad \left\{\begin{array}{l}
\displaystyle \pa_t v + v\cdot\nabla v -\Delta_\nu v+\nabla P=0, \qquad (t,x)\in\R^+\times\R^3, \\
\displaystyle \dv\, v = 0, \\
\displaystyle  v|_{t=0}=v_0=(v_0^\h,v_0^3),
\end{array}\right.
\end{equation*}
where $v=(v^\h, v^3)$ with $v^\h=(v^1,v^2)$ stands for the velocity of the incompressible fluid flow and
$P$ for the scalar pressure function,  which guarantees the divergence free condition of the velocity field,
$\Delta_{\nu}\eqdefa \Laph+\nu^2 \pa_3^2$ with $\Laph\eqdefa\pa_1^2+\pa_2^2$,
and $\nu^2$ denotes the vertical viscous coefficient.

\smallbreak

When $\nu=1,$ $(NS_\nu)$ is exactly the classical Navier-Stokes system. In the sequel, we shall always denote the system $(NS_1)$ by $(NS).$
 Whereas when $\nu=0,$ $(NS_\nu)$ reduces to the
anisotropic Navier-Stokes system arising from geophysical fluid mechanics (see \cite{CDGG}). The main motivation for us to study
Navier-Stokes system with large vertical viscous coefficient comes from the study of Navier-Stokes system on thin domains (see (2.4)
of \cite{RS93} for instance), which we shall present more details later on.

\smallbreak

In the seminal paper \cite{Leray},    Leray proved the global existence
  of finite energy weak solutions to $(NS)$. Yet the uniqueness and regularity of
such weak solutions are big open questions in the field of
mathematical fluid mechanics except the case when the initial data have special structure.
 For
instance,  with axi-symmetric initial velocity and without swirl component, Ladyzhenskaya \cite{La}  and independently Ukhovskii and Yudovich
\cite{UY}  proved the
existence of weak solution along with the uniqueness and regularity of such solution to $(NS)$. When the initial data $v_0$ has a slow space
variable, Chemin and Gallagher \cite{CG10} (see also \cite{CZ15}) can also prove the global well-posedness of such a system.

\smallbreak

While Fujita and Kato  \cite{fujitakato} proved the global well-posedness of $(NS)$ when the initial data
$v_0$ is sufficiently small in the homogeneous Sobolev space $H^{\f12}.$
This result was generalized by Cannone, Meyer and Planchon \cite{cannonemeyerplanchon} for initial data being sufficiently small
 in the homogeneous Besov
space, $ B^{-1+\frac 3p}_{p,\infty},$ with $p\in ]3,\infty[.$
 The end-point result in this direction is due to Koch and Tataru \cite{kochtataru}, where they proved the global well-posedness of
 $(NS)$ with initial data being
sufficiently small in $\text{BMO}^{-1}$,  the norm of which is determined by
\begin{equation}\label{bmonorm}
\|v\|_{\text{BMO}^{-1}}\eqdefa \|v\|_{B^{-1}_{\infty,\infty}}+\sup_{x\in\R^3,\,R>0}\f1{R^{\f32}}
\Bigl(\int_0^{R^2}\int_{B(x,R)}\bigl|e^{t\D}v(y)\bigr|^2
\,dydt\Bigr)^{\f12}.
\end{equation}
We remark that for $p\in ]3,\infty[$, there holds
$$H^{\f12}(\R^3) \hookrightarrow L^3(\R^3) \hookrightarrow B^{-1+\frac 3p}_{p,\infty}(\R^3) \hookrightarrow \text{BMO}^{-1}(\R^3)
 \hookrightarrow B^{-1}_{\infty,\infty}(\R^3),$$
and the norms to the above spaces are sclaing-invariant
under the following scaling transformation
\beq \label{S1eq2} u_\lambda(t,x)\eqdefa\lambda u(\lambda^2 t,\lambda x) \andf u_{0,\lambda}(x)\eqdefa \lambda u_0(\lambda x).\eeq
We notice that for any solution $u$ of $(NS)$ on $[0,T],$ $u_\lam$ determined by \eqref{S1eq2} is also a solution of $(NS)$ on $[0,T/\lam^2].$ We remark that the largest space, which belongs to $\cS'(\R^3)$ and the norm of which is scaling invariant under \eqref{S1eq2}, is $B^{-1}_{\infty,\infty}(\R^3)$.
Moreover, Bourgain and Pavlovi\'c \cite{BP08} proved that $(NS)$ is actually
ill-posed with initial data in $B^{-1}_{\infty,\infty}.$

\smallskip

On the other hand, by crucially using the fact that $\dive v=0,$ Zhang \cite{Zhang10}, Paicu and the second author
\cite{PZ1} improved  Fujita and Kato's result
by requiring only two components of the initial velocity being
sufficiently small in some critical Besov space
even when $\nu=0$ in $(NS_\nu)$.
Lately, Chemin and Zhang \cite{CZ} proved that if the lifespan $T^\ast$
to the Fujita-Kato solution of $(NS)$
is finite, then for any unit vector field $e$ of $\R^3$, and any
$p\in ]4,6[$, there holds
\beno
\int_0^{T^\ast}\|v\cdot e\|_{H^{\f12+\f2p}}^p\,dt=\infty.
\eeno
This result ensures that a critical norm to one component of the velocity  field controls the regularity of Fujita-Kato solution to $(NS)$.
In general, we still do not know whether or not $(NS)$ is globally well-posed
with only one component of the initial velocity being
sufficiently small. Yet we shall prove  the global well-posedness of $(NS)$ with a family of  initial data, which vary fast enough in the vertical direction 
and the third component of which  are sufficiently small, see Corollary \ref{col1} below.

Before preceding, let us recall the
anisotropic Sobolev space.

\begin{defi}\label{defanisob}
{\sl For any $s,~s'\in\R$, $H^{s,s'} $ denotes
the space of homogeneous tempered distribution~$a$  such~that
 $$
\|a\|^2_{H^{s,s'}} \eqdefa \int_{\R^3} |\xi_{\rm
h}|^{2s}|\xi_3|^{2s'} |\wh a (\xi)|^2d\xi <\infty \with \xi_{\rm
h}=(\xi_1,\xi_2).
 $$ }
\end{defi}

\no{\bf Notations:} Let us denote  $\nabla_{\rm h}\eqdefa (\pa_1,\pa_2),~\nabla_{\rm h}^\perp\eqdefa(-\p_2,\p_1),$ $\nabla_\nu\eqdefa(\pa_1,\pa_2,\nu \pa_3)$ and \begin{equation}\label{defcE}
\cE_{T}\eqdefa C\bigl([0,T[,H^{\f12}\bigr)
\bigcap L^2_{\rm{loc}}\bigl([0,T[; H^{\frac32}\bigr).
\end{equation}
For any function $a$ and any positive constant $\lam$,
we denote
$a_\lam\eqdefa a|a|^{\lam-1}.$
 $\Omega\eqdefa\curl v$  designates
the vorticity of the velocity $v,$
and $\om\eqdefa\p_1v^2-\p_2v^1,$  the third component of $\Omega$.

Our first result of this paper states as follows:

\begin{thm}\label{thmmain}
{\sl Let $v_0$ satisfy $\Omega_0=\curl v_0\in L^{\frac 32}$ and $\dive v_0=0.$
Then there exists some universal positive constant $C_1$ such that if
\begin{equation} \label{thmmaincondi}
\nu\geq C_1\bigl(M_0+M_0^{\frac14}\bigr) \with M_0\eqdefa \|\omega_0\|_{L^{\f32}}^{\f32}+\|\nabla v^3_0\|_{\hm}^2,
\end{equation}
$(NS_\nu)$ has a unique global solution $v\in\cE_{\infty}$
so that for any $t>0,$
\begin{equation}\label{thmmainestimate}
\bigl\|\om(t)\bigr\|_{L^{\f32}}^{\f32}+\|\nabla v^3(t)\|_{\hm}^2
+\int_0^t\left(\bigl\|\nabla_\nu\omss\bigr\|_{L^2}^2
+\bigl\|\nabla_\nu \nabla v^3\bigr\|_{\hm}^2\right)\, dt'
\leq 2M_0.
\end{equation}
}\end{thm}

\begin{rmk}
Due to $\dive v_0=0,$
we deduce from Sobolev inequality and Biot-Savart's law that
$$\|v^3_0\|_{H^{\frac12,0}}\leq \|v_0\|_{H^{\frac12}}
\lesssim\|\nabla v_0\|_{L^{\frac32}}
\lesssim\|\Omega_0\|_{L^{\frac32}},$$
and
\begin{equation*}\begin{split}
\|\p_3v^3_0\|_{\hm}^2&=\int_{|\xi_3|\leq|\xi_{\rm h}|}|\xi_{\rm h}|^{-1}
\bigl|\cF(\pa_3 v^3_0)(\xi)\bigr|^2 d\xi
+\int_{|\xi_{\rm h}|\leq|\xi_3|}
|\xi_{\rm h}|^{-1}
\bigl|\cF\bigl(-\divh v_0^{\rm h}\bigr)(\xi)\bigr|^2 d\xi\\
&\leq \int_{\R^3}\left(|\xi_3||\widehat{v}_0^3(\xi)|^2+|\xi_\h||\widehat{v}_0^\h(\xi)|^2\right)\,d\xi\\
& \leq\|v_0\|^2_{H^{\frac12}}\lesssim\|\Omega_0\|^2_{L^{\frac32}}.
\end{split}\end{equation*}
This implies that under the assumption of Theorem \ref{thmmain},  $M_0$ determined by \eqref{thmmaincondi} is well-defined.
Furthermore,
if $\|\Omega_0\|_{L^{\f32}}$ is sufficiently small,
 \eqref{thmmaincondi} holds for $\nu=1$.
Hence in particular,  Theorem \ref{thmmain} ensures the global well-posedness of the classical 3-D Navier-Stokes system
with  $\|\Omega_0\|_{L^{\f32}}$ being sufficiently small.
\end{rmk}

We  point out that
the main idea used to prove Theorem \ref{thmmain} can  be adapted to study the global well-posedness of the classical
3-D Navier-Stokes equations with a fast variable:
\begin{equation} \label{S1eq1}
\qquad \left\{\begin{array}{l}
\displaystyle \pa_t u + u\cdot\nabla u -\Delta u+\nabla \Pi=0, \qquad (t,x)\in\R^+\times\R^3, \\
\displaystyle \dv\, u = 0, \\
\displaystyle  u|_{t=0}=v_{0,\nu}(x_\h,\nu x_3)=(\nu v_{0}^\h(x_\h,\nu x_3),v_{0}^3(x_\h,\nu x_3)).
\end{array}\right.
\end{equation}
Let $u(t,x)\eqdefa \left(\nu v^\h(t,x_\h,\nu x_3), v^3(x_\h,\nu x_3)\right) $ and $\Pi(t,x)\eqdefa \nu P(t,x_\h,\nu x_3).$
Then $(v, P)$ verifies
\begin{equation}\label{scaleNS}
\qquad \left\{\begin{array}{l}
\displaystyle \pa_t v + \nu v\cdot\nabla v -\Delta_\nu v+\nabla_{\nu^2} P=0, \qquad (t,x)\in\R^+\times\R^3, \\
\displaystyle \na\cdot v = 0, \\
\displaystyle  v|_{t=0}=v_{0}.
\end{array}\right.
\end{equation}
\eqref{scaleNS} is closely related to the following system:
\begin{equation}\label{NSscal}
\qquad \left\{\begin{array}{l}
\displaystyle \pa_t v + v\cdot\nabla_\nu v -\Delta_\nu v+\nabla_\nu P=0, \qquad (t,x)\in\R^+\times[0,L_1]\times[0,L_2]\times [0,1], \\
\displaystyle \na_\nu\cdot v = 0, \\
\displaystyle  v|_{t=0}=v_{0,\nu}=(\nu v_{0}^\h,v_{0}^3),
\end{array}\right.
\end{equation}
which is  the rescaled Navier-Stokes system (see (2.4) of \cite{RS93}) arising from the study of 3-D Navier-Stokes system
on thin domains, $[0,L_1]\times[0,L_2]\times\bigl[0,1/\nu\bigr].$ In \cite{RS93} (see also \cite{IRS07,KZ07,TZ96}), Raugel and Sell proved the global well-posedness of \eqref{NSscal} in a periodic
domain, $[0,L_1]\times[0,L_2]\times [0,1],$ provided that $\nu$ is sufficiently large compared to the initial data. The main ideas in \cite{RS93,IRS07, KZ07,TZ96} is to decompose the solution $v$ of \eqref{NSscal} as
\beq\label{S1eq8}
v={\rm M}(v)+w \with {\rm M}(v)(t,x_\h)\eqdefa\int_0^1v(t,x_\h,x_3)\,dx_3.
\eeq
 Then the authors exploited the fact that: 2-D  Navier-Stokes
system is globally well-posed for any data in $L^2,$ and the fact that: 3-D Navier-Stokes system is globally well-posedness with
 small regular initial data, to prove that the solutions of \eqref{NSscal}
can be split as the sum of a 2-D large solution  and a 3-D small solution of $(NS)$.

We remark that in the whole space case, we do not know how to define the average of the velocity field on the vertical variable.
 Thus it is not clear how to apply the ideas in \cite{RS93,IRS07,KZ07,TZ96}
to solve \eqref{scaleNS}. Our principle result concerning the well-posedness of the system \eqref{scaleNS} is as follows:

\begin{thm}\label{thmnu}
{\sl Consider the re-scaled Navier-Stokes system \eqref{scaleNS}
 with initial data $v_0$ satisfying
$\dive v_0=0$ and $\Omega_0\in L^{\f32}.$
There exist small enough positive constants, $c_1,~c_2,$
such that if
\begin{equation}\begin{split}\label{smallcondition}
&\|\omega_0\|_{L^{\f32}}\leq c_1\nu^{-\f23}\|\nablah v^3_0\|_{L^{\f32}}^{\f12},
\qquad \nu^{-\f23}\|\nablah v^3_0\|_{L^{\f32}}\leq c_2 \andf\\
&\|\pa_3 v^3_0\|_{L^{\f32}}\|\nablah v^3_0\|_{L^{\f32}}^{\f12}\leq c_1c_2,
\qquad c_1^{\f23}\nu^{-1}\|\nablah v^3_0\|_{L^{\f32}}\leq\|\pa_3 v^3_0\|_{L^{\f32}},
\end{split}\end{equation}
or
\begin{equation}\label{smallcondition2}\begin{split}
\|\omega_0\|_{L^{\f32}}\leq c_1&\nu^{-\f23}\|\nablah v^3_0\|_{L^{\f32}}^{\f12},
\quad \nu^{-\f23}\|\nablah v^3_0\|_{L^{\f32}}\leq c_2,\\
&\andf \|\nabla v^3_0\|_{L^{\f32}}\|\nablah v^3_0\|_{L^{\f32}}^{\f12}\leq c_1c_2,
\end{split}\end{equation}
then the system \eqref{scaleNS} has
a unique global solution $v\in \cE_\infty$.
}\end{thm}

In particular, the above theorem ensures the global well-posedness of \eqref{S1eq1} provided that  the profile of the initial data
satisfying \eqref{smallcondition}
or \eqref{smallcondition2}.
For the special case when $\om_0$ vanishes, we have the following
direct consequence:
\begin{cor}\label{col1}
{\sl For any $\varphi\in W^{1,\f32}$
with $\|\nabla\varphi\|_{L^{\f32}}^2\|\nablah\varphi\|_{L^{\f32}}$ being
sufficiently small, \eqref{S1eq1} with initial data
\beq\label{ping3} u_{0,\nu}(x)=\left(-\nu\nablah\D_{\rm h}^{-1}\pa_3\varphi,\varphi\right)(x_\h,\nu x_3),\eeq
 has a unique global solution solution
provided that  $\nu$ is so large  that
$$\nu\geq C_2\|\nablah\varphi\|_{L^{\f32}}^{\f32}$$ for some positive constant $C_2$.
}\end{cor}

\begin{rmk}
For arbitrary  smooth functions $f(x_\h)$ and $g(x_3)$ with $\na_\h f\in \cS(\R^2)$ and $\pa_3g\in \cS(\R), $
we take $\varphi(x)\eqdefa \D_\h f(x_\h)g(x_3).$ Then the corresponding initial data given by \eqref{ping3} reads
\beno
u_{0,\nu}(x)=\left(-\nu \na_\h f(x_\h)\p_3g(\nu x_3), \D_\h f(x_\h)g(\nu x_3)\right).
\eeno
In particular, let us take $f$
so that \beq
\label{ping4} |\na_\h f(x_\h)|\geq 2\quad\mbox{ in some small neighborhood of}\quad x_\h=0.
\eeq  Then we can select
$\nu$ so large that
$\bigl|e^{t\D_{\rm h}}\na_\h f(x_\h)\bigr|\geq 1$ for any
$(t,x_\h)\in P_{\nu^{-1}},$ where $P_R$ denotes $[0,R^2]\times B_R$.
Furthermore,  by virtue of \eqref{bmonorm}, we have
\begin{align*}
\|u_{0,\nu}^\h\|_{\text{BMO}^{-1}}&\geq\nu^{\f32}
\Bigl(\int_{P_{\nu^{-1}}}\bigl|e^{t\D}\bigl(\nu \na_\h f(x_\h)\p_3 g(\nu x_3)
\bigr)\bigr|^2\,dxdt\Bigr)^{\f12}\\
&=\nu^{\f52}\Bigl(\int_{P_{\nu^{-1}}}\bigl|(e^{t\D_\h}\na_\h f)(x_\h)
\cdot(e^{\nu^2t\pa_3^2}\p_3 g)(\nu x_3)\bigr|^2
\,dxdt\Bigr)^{\f12}\\
&\geq C\nu^{\f32}\Bigl(\int_0^{\nu^{-2}}\int_{-\f12\nu^{-1}}^{\f12\nu^{-1}}
\bigl|(e^{\nu^2t\pa_3^2}\p_3 g)(\nu x_3)\bigr|^2
\,dx_3dt\Bigr)^{\f12}\\
&=C\bigl\|e^{t\pa_3^2}\p_3 g\bigr\|_{L^2([0,1]\times[-\f12,\f12])}.
\end{align*}
Notice that for  any smooth function, $g(x_3),$ with $\pa_3 g\in\cS(\R),$ we can always find smooth function, $f(x_\h),$ with $\na_\h f\in \cS(\R^2)$ so that \eqref{ping4} holds and $\|\nabla\varphi\|_{L^{\f32}}^2\|\nablah\varphi\|_{L^{\f32}}$ is sufficiently small (for instance $\na_\h f(x_\h)=\chi(R x_\h)$ for $R$ large enough,  where
$\chi=(\chi_1, \chi_2)\in C_0^\infty(\R^2)$ and $\chi$ satisfies $|\chi(x_\h)|\geq 2$ for $x_\h$ near $0.$)
Hence Theorem \ref{thmnu} and Corollary \ref{col1} can not deduced from the end-point result in \cite{kochtataru}.
\end{rmk}

\no{\it Sketch of the paper}.
Motivated by \cite{CZ}, we first reformulate $(NS_\nu)$  as
\begin{equation}\label{omegav3}
\qquad \left\{\begin{array}{l}
\displaystyle \pa_t\omega+v\cdot\nabla\omega -\Delta_{\nu}\omega= \partial_3v^3\omega +\pa_2v^3\pa_3v^1-\pa_1v^3\pa_3v^2,\\
\displaystyle \pa_t v^3 + v\cdot\nabla v^3 -\Delta_{\nu} v^3=
-\partial_3\D^{-1} \Bigl(\sum_{\ell,m=1}^3 \partial_\ell
v^m\partial_mv^\ell\Bigr), \\
\displaystyle  \omega|_{t=0}=\omega_0,\ v^3|_{t=0}=v^3_0.
\end{array}\right.
\end{equation}
Then due to $\divh \vh=-\pa_3 v^3,$ given $(\om, v^3),$  by Biot-Savart's law, we  write
\begin{equation}\label{helmdecom}
v^{\rm h}=\vcurl+\vdiv,\quad \mbox{where}\quad \vcurl
\eqdefa\nablah^\perp \D_{\rm h}^{-1} \om \quad \mbox{and} \quad
 \vdiv\eqdefa -\nablah\D_{\rm h}^{-1}\partial_3v^3.
\end{equation}

Let us recall the following results from \cite{CZ}:

\begin{thm}\label{thm1.4ofCZ}
{\sl Let us consider an initial data $v_0$ with vorticity $\Omega_0\in L^{\frac32}$.
Then a unique maximal solution $v$ of $(NS_\nu)$ exists in the space
$\cE_{T^\ast}$ for some
maximal existing time $T^\ast>0$, and this solution satisfies
$$\Omega\eqdefa\curl v\in C\bigl([0,T^\ast[,L^{\f32}\bigr)
\quad \mbox{and}\quad |\nabla\Omega|\cdot|\Omega|^{-\f14}\in
L^2_{\rm{loc}}\bigl([0,T^\ast[; L^2\bigr)$$
Moreover, if $T^\ast<\infty$,  for any $p\in]4,6[$, we have
\begin{equation}\label{blowupCZ5}
\int_0^{T^\ast}\|v^3(t)\|_{H^{\frac12+\frac2p}}^p\,dt=\infty.
\end{equation}
}\end{thm}

According to Theorem \ref{thm1.4ofCZ},
 in order to prove Theorem \ref{thmmain},
it remains  to verify that
\eqref{blowupCZ5} can never be satisfied under the assumption \eqref{thmmaincondi}. It is easy to observe that
\begin{equation}\begin{split}\label{v3anisotropic}
\|v^3\|_{H^{\frac12+\frac2p}}&\leq
\|v^3\|_{H^{\frac12+\frac2p,0}}+\|v^3\|_{H^{0,\frac12+\frac2p}}\\
&\leq\|v^3\|_{H^{\frac12,0}}^{1-\frac2p}
\|\nabla_\h v^3\|_{H^{\frac12,0}}^{\frac2p}
+\|\pa_3 v^3\|_{\hm}^{\frac{1}{2}-\frac2p}\|v^3\|_{\hh}^{\f12}
\|\pa_3^2 v^3\|_{\hm}^{\frac 2p}.
\end{split}\end{equation}
It reduces to derive the $\hh$ estimate for $v^3$
and the $\hm$ estimate of $\pa_3v^3$, that is,
the $\hm$ estimate of $\nabla v^3$.
In view of  \eqref{omegav3} and \eqref{helmdecom}, in order to close the estimates,
we also need the $L^{\f32}$ estimate for $\omega.$
As a matter of fact,  under the assumption \eqref{thmmaincondi},
we can indeed achieve the estimate \eqref{thmmainestimate}.
This in turn shows that $\|v^3\|_{L^p_T(H^{\frac12+\frac2p})}$
is finite for any $p\in ]4,6[$ and any $T<\infty.$
  Then theorem \ref{thmmain} follows from Theorem \ref{thm1.4ofCZ}.

Along the same line, we can equivalently
reformulate the system \eqref{scaleNS} as
\begin{equation}\label{omegav3nu}
\qquad \left\{\begin{array}{l}
\displaystyle \pa_t\omega+\nu v\cdot\nabla\omega -\Delta_{\nu}\omega= \nu\bigl(\partial_3v^3\omega +\pa_2v^3\pa_3v^1-\pa_1v^3\pa_3v^2\bigr),\\
\displaystyle \pa_t v^3 +\nu v\cdot\nabla v^3-\Delta_{\nu} v^3=
-\nu^2\partial_3\D^{-1}\Bigl(\sum_{\ell,m=1}^3\nu \partial_\ell
v^m\partial_mv^\ell\Bigr), \\
\displaystyle \omega|_{t=0}=\omega_0,\ v^3|_{t=0}=v^3_0.
\end{array}\right.
\end{equation}
Obviously, the difference between the systems \eqref{omegav3} and \eqref{omegav3nu}
is that there appears powers of  $\nu$ in the front of the quadric terms in \eqref{omegav3nu},
which makes it more difficult to perform the uniform estimates. Indeed the most dangerous term
is $-\nu^2\partial_3\D^{-1} \bigl(\sum_{\ell,m=1}^2\nu \partial_\ell
v_{\curl}^m\partial_m v_{\curl}^\ell\bigr)$, which is more or less the same
as $-\nu^3\partial_3\D^{-1}\left(\omega^2\right)$. Thus if we want to close the previous estimates
for sufficiently large $\nu$, it seems necessary that $\omega$ should
be small for all time.
That is the reason why we need  the smallness condition
\eqref{smallcondition} in Theorem \ref{thmnu}.

\setcounter{equation}{0}
\section{Preliminaries}

We first recall some basic facts on anisotropic Littlewood-Paley theory
  from \cite{BCD}:
\begin{equation}\begin{split}\label{defparaproduct}
&\Delta_ja=\cF^{-1}(\varphi(2^{-j}|\xi|)\widehat{a}),
 \quad \Delta_k^{\rm h}a=\cF^{-1}(\varphi(2^{-k}|\xi_{\rm h}|)\widehat{a}),
 \quad \Delta_\ell^{\rm v}a =\cF^{-1}(\varphi(2^{-\ell}|\xi_3|)\widehat{a}),\\
&S_ja=\cF^{-1}(\chi(2^{-j}|\xi|)\widehat{a}),
\quad S^{\rm h}_ka=\cF^{-1}(\chi(2^{-k}|\xi_{\rm h}|)\widehat{a}),
\quad\ S^{\rm v}_\ell a =\cF^{-1}(\chi(2^{-\ell}|\xi_3|)\widehat{a}),
\end{split}\end{equation}
where $\xi_{\rm h}=(\xi_1,\xi_2),$ $\cF a$ and
$\widehat{a}$ denote the Fourier transform of the distribution $a,$
$\chi(\tau)  $ and~$\varphi(\tau)$ are smooth functions such that
 \beno
&&\Supp \varphi \subset \Bigl\{\tau \in \R\,/\  \ \frac34 \leq
|\tau| \leq \frac83 \Bigr\}\andf \  \ \forall
 \tau>0\,,\ \sum_{j\in\Z}\varphi(2^{-j}\tau)=1,\\
&&\Supp \chi \subset \Bigl\{\tau \in \R\,/\  \ \ |\tau|  \leq
\frac43 \Bigr\}\quad \ \ \ \andf \  \ \, \chi(\tau)+ \sum_{j\geq
0}\varphi(2^{-j}\tau)=1.
 \eeno

\begin{defi}\label{anibesov}
{\sl Let us define the space $\bigl(B^{s_1}_{p,q_1}\bigr)_{\rm
h}\bigl(B^{s_2}_{p,q_2}\bigr)_{\rm v}$ (with  usual adaptation when $q_1$ or $q_2$ equal $\infty$) as the space of
homogenous  tempered distributions $u$ so that
$$
\|u\|_{\bigl(B^{s_1}_{p,q_1}\bigr)_{\rm
h}\bigl(B^{s_2}_{p,q_2}\bigr)_{\rm v}}\eqdefa \biggl(\sum_{k\in\Z}
2^{q_1ks_1} \Bigl(\sum_{\ell\in\Z}2^{q_2\ell s_2}\|\D_k^{\rm
h}\D_\ell^{\rm
v}u\|_{L^p}^{q_2}\Bigr)^{{q_1}/{q_2}}\biggr)^{1/{q_1}}
$$ is finite.
For the special case when $q_1=q_2=q$, we shall denote it briefly by
$B^{s_1,s_2}_{p,q}$.
}
\end{defi}

We remark that $B^{s_1,s_2}_{2,2}$ coincides with the classical
anisotropic Sobolev space $H^{s_1,s_2}$.

 \begin{lem}[A spacial case of Lemma 4.5 in \cite{CZ}]\label{lemproductlaw}
{\sl For any $s_1<1,~s_2\leq1$ with $s_1+s_2>0,$ and for
any $r_1<\f12,~r_2\leq\f12$ with $r_1+r_2>0$, we have
\begin{equation}\label{lemproduct2}
\|ab\|_{H^{s_1+s_2-1,r_1+r_2-\f12}}\lesssim\|a\|_{H^{s_1,r_1}}
\|b\|_{B^{s_2,r_2}_{2,1}}.
\end{equation}
When $s_1, s_2<1$ and $r_1, r_2<\frac12,$  one has
\ben
\label{lemproduct1}
\|ab\|_{H^{s_1+s_2-1,r_1+r_2-\f12}}&\lesssim&\|a\|_{H^{s_1,r_1}}
\|b\|_{H^{s_2,r_2}},\\
\label{lemproduct3}
\|ab\|_{(B^{-1}_{2,\infty})_\h(H^{r_1+r_2-\f12})_\v}
&\lesssim&\|a\|_{H^{s_1,r_1}}\|b\|_{H^{-s_1,r_2}}.
\een

}\end{lem}

Let us recall the following anisotropic
Bernstein type lemma from \cite{CZ1, Pa02}:

\begin{lem}
\label{lemBern}
{\sl Let $\cB_{\h}$ (resp.~$\cB_{\rm v}$) a ball
of~$\R^2_{\h}$ (resp.~$\R_{\rm v}$), and~$\cC_{\h}$ (resp.~$\cC_{\rm v}$) a
ring of~$\R^2_{\h}$ (resp.~$\R_{\rm
v}$); let~$1\leq p_2\leq p_1\leq
\infty$ and ~$1\leq q_2\leq q_1\leq \infty.$ Then there holds:

\smallbreak\noindent If the support of~$\wh a$ is included
in~$2^k\cB_{\h}$, then
\[
\|\partial_{x_{\rm h}}^\alpha a\|_{L^{p_1}_{\rm h}(L^{q_1}_{\rm v})}
\lesssim 2^{k\left(|\al|+2\left(1/{p_2}-1/{p_1}\right)\right)}
\|a\|_{L^{p_2}_{\rm h}(L^{q_1}_{\rm v})}.
\]
If the support of~$\wh a$ is included in~$2^\ell\cB_{\rm v}$, then
\[
\|\partial_{x_3}^\beta a\|_{L^{p_1}_{\rm h}(L^{q_1}_{\rm v})}
\lesssim 2^{\ell\left(\beta+(1/{q_2}-1/{q_1})\right)} \|
a\|_{L^{p_1}_{\rm h}(L^{q_2}_{\rm v})}.
\]
If the support of~$\wh a$ is included in~$2^k\cC_{\h}$, then
\[
\|a\|_{L^{p_1}_{\rm h}(L^{q_1}_{\rm v})} \lesssim
2^{-kN}\sup_{|\al|=N} \|\partial_{x_{\rm h}}^\al a\|_{L^{p_1}_{\rm
h}(L^{q_1}_{\rm v})}.
\]
If the support of~$\wh a$ is included in~$2^\ell\cC_{\rm v}$, then
\[
\|a\|_{L^{p_1}_{\rm h}(L^{q_1}_{\rm v})} \lesssim 2^{-\ell N}
\|\partial_{x_3}^N a\|_{L^{p_1}_{\rm h}(L^{q_1}_{\rm v})}.
\] Here and in all that follows, $a\lesssim b$ means that
$a\leq Cb$ for some uniform constant $C$.
}
\end{lem}

The following Troisi inequality in \cite{Tro} will play an important role
in what follows.

\begin{lem}\label{lemTro}
{\sl Let $1\leq q_i<\infty~(i=1,\cdots,d)$ with $\sum_{i=1}^d q_i^{-1}>1$,
and $p=\f{d}{\sum_{i=1}^d q_i^{-1}-1}$.
Then for any $a\in C_0^\infty(\R^d)$, there holds
$$\|a\|_{L^p(\R^d)}\leq C\prod_{i=1}^d
\|\pa_i a\|_{L^{q_i}(\R^d)}^{\f1d}.$$
In the particular case when $d=3$, we have
\begin{equation}\label{ineqtro}
\|a\|_{L^6(\R^3)}\leq C\|\nablah a\|_{L^2(\R^3)}^{\f23}
\|\pa_3 a\|_{L^2(\R^3)}^{\f13}.
\end{equation}
}\end{lem}

Applying Lemma \ref{lemTro} leads to  the following  interpolation inequality:

\begin{lem}
\label{lemcontrolomss}
{\sl For  $i=1,2,3$, we have
\begin{equation}\label{2.3}
 \|\pa_i \omega\|_{L^{\frac32}} \leq C \bigl\|\pa_i \omega_{\frac34}\bigr\|_{L^2}
 \bigl\|\omega_{\frac34}\bigr\|_{L^2}^{\frac 13},
\end{equation}
and
\begin{equation}\label{2.4}
\begin{split}
&\|\omega\|_{L^{\frac95}} \leq C \bigl\|\omega_{\frac34}\bigr\|_{L^2}\bigl\|\nablah\omega_{\frac34}\bigr\|_{L^2}^{\frac29}
\bigl\|\pa_3 \omega_{\frac34}\bigr\|_{L^2}^{\frac19};\\
&\|\pa_i \omega\|_{L^{\frac95}} \leq C \bigl\|\pa_i \omega_{\frac34}\bigr\|_{L^2}\bigl\|\nablah\omega_{\frac34}\bigr\|_{L^2}^{\frac29}
\bigl\|\pa_3 \omega_{\frac34}\bigr\|_{L^2}^{\frac19}.
\end{split}
 \end{equation}
}\end{lem}
\begin{proof}
Note that $|\pa_i\omega|=\frac43\bigl|\pa_i\omss\bigr|\cdot|\omega|^{\frac14}$,
we get, by applying Holder's inequality, that
$$\|\pa_i \omega\|_{L^{\frac32}}\lesssim\bigl\|\pa_i \omega_{\frac34}\bigr\|_{L^2}\bigl\||\omega|^{\frac14}\bigr\|_{L^6}
\lesssim\bigl\|\pa_i \omega_{\frac34}\bigr\|_{L^2}
\bigl\|\omss\bigr\|_{L^2}^{\frac 13},$$
which is the desired estimate \eqref{2.3}.

 Along the same line, we have
\beno
\|\pa_i \omega\|_{L^{\frac95}}\lesssim\bigl\|\pa_i \omega_{\frac34}\bigr\|_{L^2}\bigl\||\omega|^{\frac14}\bigr\|_{L^{18}}
\lesssim\bigl\|\pa_i \omega_{\frac34}\bigr\|_{L^2}
\bigl\|\omega_{\frac34}\bigr\|_{L^6}^{\frac13}, \eeno
 and
 \beno
\|\omega\|_{L^{\frac95}}
\lesssim\bigl\|\omega_{\frac34}\bigr\|_{L^2}
\bigl\|\omega_{\frac34}\bigr\|_{L^6}^{\frac13}.
\eeno
Inserting  \eqref{ineqtro} into the above inequalities  leads to \eqref{2.4}.
\end{proof}

As an application of the above basic facts, we shall present the estimate of $\omega$ in the
anisotropic Sobolev spaces.

\begin{prop}\label{lemomega}
{\sl Let $\omega\in L^{\f32}$ with $\nabla\omss\in L^2$. Let  $\cD$ be the convex hull of the following points:
$\bigl(-\f19,\f{17}{18}\bigr),~\bigl(\f89,-\f1{18}\bigr),~\bigl(-\f13,\f56\bigr),
~\bigl(\f23,-\f16\bigr),$ and $\bigl(-\f13,-\f16\bigr)$, which can also  be characterized by
$$\cD\eqdefa\Bigl\{(s_1,s_2)\in\R^2\,:\,s_1\geq-\f13,~s_2\geq-\f16,~s_1+s_2\leq\f56,~
-2\leq s_1-2s_2\leq 1\Bigr\}.$$
Then for any $(s_1,s_2)\in \cD$, we have
\begin{equation}\label{2.7}
\|\omega\|_{H^{s_1,s_2}}\leq \bigl\|\omss\bigr\|_{L^2}^{\f56-s_1-s_2}
\bigl\|\nablah \omss\bigr\|_{L^2}^{\f13+s_1}
\bigl\|\pa_3 \omss\bigr\|_{L^2}^{\f16+s_2}.
\end{equation}
If $(s_1,s_2)$ is an inner point of $\cD$, one has
\begin{equation}\label{2.8}
\|\omega\|_{B^{s_1,s_2}_{2,1}}\leq \bigl\|\omss\bigr\|_{L^2}^{\f56-s_1-s_2}
\bigl\|\nablah \omss\bigr\|_{L^2}^{\f13+s_1}
\bigl\|\pa_3 \omss\bigr\|_{L^2}^{\f16+s_2},
\end{equation}
}\end{prop}
\begin{proof}
Thanks to Theorem 2.40 of \cite{BCD}, which claims that  $L^p\hookrightarrow B^0_{p,2}$
for any $p\in[1,2],$  we get, by applying Minkowski's inequality, that
\begin{align*}
\|a\|_{B^{0,0}_{p,2}}&=\Bigl(\sum_{k,\ell\in\Z}\|\dhk\dvl a\|_{L^{p}}^2
\Bigr)^{\f12}\\
&\lesssim\Bigl(\sum_{k\in\Z}\bigl\|\bigl(\sum_{\ell\in\Z}
\|\dhk\dvl a\|_{L_\v^{p}}^2\bigr)^{\f12}\bigr\|_{L_\h^{p}}^2
\Bigr)^{\f12}\\
&\lesssim\Bigl(\sum_{k\in\Z}\bigl\|
\|\dhk a\|_{L_\v^{p}}\bigr\|_{L_\h^{p}}^2
\Bigr)^{\f12}\\
&\lesssim\Bigl\|\bigl(\sum_{k\in\Z}
\|\dhk a\|_{L_\h^{p}}^2
\bigr)^{\f12}\Bigr\|_{L_\v^{p}}\lesssim\|a\|_{L^{p}},
\end{align*}
from which and Lemma \ref{lemBern},  we deduce that for any $s\in[0,1]$
\begin{equation*}
\begin{split}
&\|\omega\|_{H^{-\f13+s,-\f1{6}}}\lesssim\|\omega\|_{B^{s,0}_{\f32,2}}
\lesssim\|\nablah\omega\|_{B^{0,0}_{\f32,2}}^s
\|\omega\|_{B^{0,0}_{\f32,2}}^{1-s}
\lesssim\|\nablah\omega\|_{L^{\f32}}^s
\|\omega\|_{L^{\f32}}^{1-s};\\
&\|\omega\|_{H^{-\f19+s,-\f1{18}}}\lesssim\|\omega\|_{B^{s,0}_{\f95,2}}
\lesssim\|\nablah\omega\|_{B^{0,0}_{\f95,2}}^s
\|\omega\|_{B^{0,0}_{\f95,2}}^{1-s}
\lesssim\|\nablah\omega\|_{L^{\f95}}^s
\|\omega\|_{L^{\f95}}^{1-s}.
\end{split}
\end{equation*}
Inserting the estimates \eqref{2.3} and \eqref{2.4} into the above
inequalities gives rise to
\begin{equation}\label{2.9}
\begin{split}
&\|\omega\|_{H^{-\f13+s,-\f1{6}}}
\lesssim\bigl\|\omss\bigr\|_{L^2}^{\f43-s}
\bigl\|\nablah \omss\bigr\|_{L^2}^{s};\\
&\|\omega\|_{H^{-\f19+s,-\f1{18}}}
\lesssim\bigl\|\omss\bigr\|_{L^2}^{1-s}
\bigl\|\nablah \omss\bigr\|_{L^2}^{\f29+s}
\bigl\|\pa_3 \omss\bigr\|_{L^2}^{\f19}.
\end{split}
\end{equation}
Exactly along the same line, we  deduce that for any $s\in[0,1]$
\begin{equation}\label{2.10}
\begin{split}
&\|\omega\|_{H^{-\f13,-\f1{6}+s}}
\lesssim\bigl\|\omss\bigr\|_{L^2}^{\f43-s}
\bigl\|\pa_3 \omss\bigr\|_{L^2}^{s};\\
&\|\omega\|_{H^{-\f19,-\f1{18}+s}}
\lesssim\bigl\|\omss\bigr\|_{L^2}^{1-s}
\bigl\|\nablah \omss\bigr\|_{L^2}^{\f29}
\bigl\|\pa_3 \omss\bigr\|_{L^2}^{\f19+s}.
\end{split}
\end{equation}
By interpolating the inequalities \eqref{2.9} and \eqref{2.10}, we obtain that
 for any $r_1,~r_2\in[0,1]$ with $r_1+r_2\leq 1$,
 \begin{equation}\begin{split}\label{2.12}
\|\omega\|_{H^{-\f13+r_1,-\f1{6}+r_2}}
&\leq\|\omega\|_{H^{-\f13+r_1+r_2,-\f1{6}}}^{\f{r_1}{r_1+r_2}}
\|\omega\|_{H^{-\f13,-\f1{6}+r_1+r_2}}^{\f{r_2}{r_1+r_2}}\\
&\lesssim\bigl\|\omss\bigr\|_{L^2}^{\f43-r_1-r_2}
\bigl\|\nablah \omss\bigr\|_{L^2}^{r_1}
\bigl\|\pa_3 \omss\bigr\|_{L^2}^{r_2};
\end{split}\end{equation}
and
 \begin{equation}\begin{split}\label{2.11}
\|\omega\|_{H^{-\f19+r_1,-\f1{18}+r_2}}
&\leq\|\omega\|_{H^{-\f19+r_1+r_2,-\f1{18}}}^{\f{r_1}{r_1+r_2}}
\|\omega\|_{H^{-\f19,-\f1{18}+r_1+r_2}}^{\f{r_2}{r_1+r_2}}\\
&\lesssim\bigl\|\omss\bigr\|_{L^2}^{1-r_1-r_2}
\bigl\|\nablah \omss\bigr\|_{L^2}^{\f29+r_1}
\bigl\|\pa_3 \omss\bigr\|_{L^2}^{\f19+r_2}.
\end{split}\end{equation}
The estimates \eqref{2.12} and \eqref{2.11} show that \eqref{2.7}
holds for $(s_1,s_2)\in \Bigl\{\ \bigl(-\f19,\f{17}{18}\bigr),~\bigl(\f89,-\f1{18}\bigr)$,\\
$\bigl(-\f13,\f56\bigr)$,~$\bigl(\f23,-\f16\bigr),~\bigl(-\f13,-\f16\bigr)\ \Bigr\}.$
Then  \eqref{2.7} for any $(s_1,s_2)\in \cD$ follows from a classical interpolation argument.

In order to prove \eqref{2.8}, let us notice that for any integers $N, M_1, M_2$ we have
\beno
\begin{split}
\|\omega\|_{B^{s_1,s_2}_{2,1}}=&\Bigl(\sum_{\substack{k\leq N\\ \ell\leq M_1}}+\sum_{\substack{k\leq N\\ \ell> M_1}}+\sum_{\substack{k> N\\ \ell\leq M_2}}
+\sum_{\substack{k> N\\ \ell> M_2}}\Bigr)2^{ks_1}2^{\ell s_2}\|\D_k^\h\D_\ell^\v \om\|_{L^2}\\
\lesssim & 2^{N\d} 2^{M_1\d}\|\omega\|_{H^{s_1-\delta,s_2-\delta}}+ 2^{N\d}2^{-M_1\d}\|\omega\|_{H^{s_1-\delta,s_2+\delta}}\\
&+2^{-N\d}2^{M_2\d}\|\omega\|_{H^{s_1+\delta,s_2-\delta}}+ 2^{-N\d}2^{-M_2\d}\|\omega\|_{H^{s_1+\delta,s_2+\delta}}.
\end{split}
\eeno
Taking $M_1, M_2$ in the above inequality so that
\beno
2^{2M_1\d}\sim\f{\|\omega\|_{H^{s_1-\delta,s_2+\delta}}}{\|\omega\|_{H^{s_1-\delta,s_2-\delta}}} \andf
2^{2M_2\d}\sim \f{\|\omega\|_{H^{s_1+\delta,s_2+\delta}}}{\|\omega\|_{H^{s_1+\delta,s_2-\delta}}},
\eeno gives rise to
\beno
\begin{split}
\|\omega\|_{B^{s_1,s_2}_{2,1}}\lesssim 2^{N\d}\|\omega\|_{H^{s_1+\delta,s_2+\delta}}^{\f12}
\|\omega\|_{H^{s_1+\delta,s_2-\delta}}^{\f12}+2^{-N\d}\|\omega\|_{H^{s_1-\delta,s_2+\delta}}^{\f12}
\|\omega\|_{H^{s_1-\delta,s_2-\delta}}^{\f12}.
\end{split}
\eeno
Taking $N$ in the above inequality so that
\beno
2^{2N\d}\sim \frac{\|\omega\|_{H^{s_1-\delta,s_2+\delta}}^{\f12}
\|\omega\|_{H^{s_1-\delta,s_2-\delta}}^{\f12}}{\|\omega\|_{H^{s_1+\delta,s_2+\delta}}^{\f12}
\|\omega\|_{H^{s_1+\delta,s_2-\delta}}^{\f12}}
\eeno
leads to
\beq \label{bs1s2} \|\omega\|_{B^{s_1,s_2}_{2,1}}\lesssim\|\omega\|_{H^{s_1+\delta,s_2+\delta}}^{\f14}
\|\omega\|_{H^{s_1+\delta,s_2-\delta}}^{\f14}\|\omega\|_{H^{s_1-\delta,s_2+\delta}}^{\f14}
\|\omega\|_{H^{s_1-\delta,s_2-\delta}}^{\f14}.\eeq

On the other hand,  when $(s_1,s_2)$ is an inner point of $\cD$,  we can find some
$\delta>0$ so that the ball with center $(s_1,s_2)$
and radius $2\delta$ is still contained in $\cD$.
 Then \eqref{2.8} follows by inserting \eqref{2.7} into \eqref{bs1s2}. This completes the proof of this lemma.
\end{proof}

\setcounter{equation}{0}
\section{The proofs of Theorem \ref{thmmain}}
This section is devoted to the proof of Theorem \ref{thmmain}.
And the strategy is to verify that, the necessary condition
for finite time blow-up criteria, \eqref{blowupCZ5},  can never be satisfied for the local Fujita-Kato solutions of $(NS_\nu)$ under
the assumption of \eqref{thmmaincondi}.

\subsection{{\it A priori} estimates}
Let us first derive the estimates for $\bigl\|\omss(t)\bigr\|_{L^2}^2$
and $\|\nabla v^3(t)\|_{\hm}^2$,
which will be essential in our proof.
To do it, we denote
\begin{equation}\begin{split}\label{defMN}
M(t)\eqdefa\bigl\|\omss(t)&\bigr\|_{L^2}^2+\|\nabla v^3(t)\|_{\hm}^2,
\quad N(t)\eqdefa\bigl\|\pa_3\omss(t)\bigr\|_{L^2}^2
+\|\pa_3\nabla v^3(t)\|_{\hm}^2,\\
&\mbox{and}\quad\wt N(t)\eqdefa\bigl\|\nablah\omss(t)\bigr\|_{L^2}^2
+\|\nablah\nabla v^3(t)\|_{\hm}^2.
\end{split}\end{equation}
We emphasize the fact that
\beno \|\nabla v^3(t)\|_{\hm}^2=\|v^3(t)\|_{H^{\f12,0}}^2+\|\p_3v^3(t)\|_{\hm}^2. \eeno

\begin{prop}\label{propestimateomss}
{\sl Let $v=(v^\h,v^3)$ be a smooth enough solution of $(NS_\nu)$ on $[0,T^\ast[$.
 Then for any $t<T^\ast$, there holds
\begin{equation}\label{estimateomss}
\frac{d}{dt}\bigl\|\omss(t)\bigr\|_{L^2}^2+
 \frac43\bigl\|\nabla_\nu\omss(t)\bigr\|_{L^2}^2
\leq\frac29\wt N+C\bigl(M^{\f23}+M^{\f12}\bigr)N.
\end{equation}
}\end{prop}

\begin{proof}
By taking $L^{2}$ inner product of the  $\omega$ equation in \eqref{omegav3} with $\omega_{\f12}$, we get
\begin{equation}\label{estimateomega}
\frac{d}{dt}\bigl\|\omss(t) \bigr\|_{L^2}^2+
\frac43 \bigl\|\nabla_\nu\omss(t)\bigr\|_{L^2}^2
=\frac32\int_{\R^3}\pa_3v^3|\omega|^{\frac32}\,dx
+\frac32\int_{\R^3}(\pa_2v^3\pa_3v^1-\pa_1v^3\pa_3v^2)\omega_{\frac12}\,dx.
\end{equation}
We first get, by using integration by parts and then H\"{o}lder's inequality, that
\begin{equation}\begin{split}\label{pa3v3om32}
\bigl|\int_{\R^3}\pa_3v^3|\omega|^{\frac32}\,dx\bigr|
&\leq \f32\int_{\R^3}|v^3||\pa_3\omega||\omega|^{\frac12}\,dx\\
&\leq \f32\|v^3\|_{L^{\frac92}_{\rm{h}}(L^6_{\rm{v}})}\|\pa_3\omega\|_{L^{\frac32}}
\bigl\|\omss\bigr\|_{L^6_{\rm{h}}(L^4_{\rm{v}})}^{\frac23}.
\end{split}\end{equation}
While it is easy to observe  that
\beno
\begin{split}
\|a\|_{H^{\frac23,\frac14}}^2=&\int_{\R^3}1^{\f1{12}}\bigl(|\xi_\h|^2\bigr)^{\f23}\bigl(|\xi_3|^2\bigr)^{\f14}|\hat{a}(\xi)|^2\,d\xi\\
\leq &\Bigl(\int_{\R^3}|\hat{a}(\xi)|^2\,d\xi\Bigr)^{\f1{12}}\Bigl(\int_{\R^3}|\xi_\h|^2|\hat{a}(\xi)|^2\,d\xi\Bigr)^{\f2{3}}
\Bigl(\int_{\R^3}|\xi_3|^2|\hat{a}(\xi)|^2\,d\xi\Bigr)^{\f1{4}}\\
\leq& \|a\|_{L^2}^{\frac1{6}}
\|\nablah a\|_{L^2}^{\frac43}\|\pa_3a\|_{L^2}^{\frac12},
\end{split}
\eeno
which yields
\beno
\|a\|_{H^{\frac23,\frac14}}\leq \|a\|_{L^2}^{\frac1{12}}
\|\nablah a\|_{L^2}^{\frac23}\|\pa_3a\|_{L^2}^{\frac14}.
\eeno
Similarly, we have
\beno
 \|a\|_{H^{\frac59,\frac13}}
\leq \|a\|_{H^{\frac12,0}}^{\frac{11}{18}}
\|\nablah a\|_{H^{\frac12,0}}^{\frac1{18}}
\|\pa_3 a\|_{\hh}^{\frac13}.
\eeno
So that it follows from Sobolev inequality that
\begin{align*}
\bigl\|\omss\bigr\|_{L^6_{\rm{h}}(L^4_{\rm{v}})}\lesssim&\bigl\|\omss\bigr\|_{H^{\frac23,\frac14}}
\lesssim\bigl\|\omss\bigr\|_{L^2}^{\frac1{12}}
\bigl\|\nablah\omss\bigr\|_{L^2}^{\frac23}\bigl\|\pa_3\omss\bigr\|_{L^2}^{\frac14},\\
\|v^3\|_{L^{\frac92}_{\rm{h}}(L^6_{\rm{v}})}\lesssim &\|v^3\|_{H^{\frac59,\frac13}}
\lesssim \|v^3\|_{H^{\frac12,0}}^{\frac{11}{18}}
\|\nablah v^3\|_{H^{\frac12,0}}^{\frac1{18}}
\|\pa_3v^3\|_{\hh}^{\frac13}.
\end{align*}
Substituting the above inequalities and \eqref{2.3} into \eqref{pa3v3om32}, we obtain
\begin{equation}\begin{split}\label{nuomss1}
\bigl|\int_{\R^3}\pa_3v^3|\omega|^{\frac32}\,dx\bigr|
&\lesssim \bigl\|\omss\bigr\|_{L^2}^{\frac{7}{18}}\|v^3\|_{H^{\frac12,0}}^{\frac{11}{18}}
\bigl\|\nablah\omss\bigr\|_{L^2}^{\frac49}\|\nablah v^3\|_{H^{\frac12,0}}^{\frac1{18}}
\bigl\|\pa_3\omss\bigr\|_{L^2}^{\frac76}\|\pa_3v^3\|_{\hh}^{\frac13}\\
&\lesssim M^{\f12}\wt N^{\f14}N^{\f34}.
\end{split}\end{equation}
Applying Young's inequality gives
\begin{equation}\begin{split}\label{estimateomega1}
\bigl|\int_{\R^3}\pa_3v^3|\omega|^{\frac32}\,dx\bigr|
\leq \frac1{9}\wt N
+C M^{\f23}N.
\end{split}\end{equation}

To deal with the second term on the right side of \eqref{estimateomega}, we need the following lemma,
which we admit for the time being.

\begin{lem}\label{lemomhalf}
{\sl For any Schwartz functions $f,~g$ and $\omega$, there holds
\begin{align*}
\Bigl|\int_{\R^3}\nablah\Delta_{\rm{h}}^{-1}f\cdot\nablah g
\bigl| \omega_{\frac12}\,dx\Bigr|\lesssim
\min\Bigl\{\|f\|_{L^{\frac32}}\|\pa_3 g&\|_{\hh}^{\f16},
\|f\|_{\hm}\|\nablah g\|_{\hh}^{\f16}\Bigr\}\\
&\times\|g\|_{\hh}^{\f13}\|\nablah g\|_{\hh}^{\f12}
\bigl\|\omss\bigr\|_{L^2}^{\frac13}\bigl\|\pa_3\omss\bigr\|_{L^2}^{\frac13}.
\end{align*}
}\end{lem}

By virtue of \eqref{helmdecom}, we have
$$\pa_3\vh=\nablah^\perp \D_{\rm h}^{-1} \pa_3\om
-\nablah\D_{\rm h}^{-1}\partial_3^2 v^3.$$
Applying Lemma \ref{lemomhalf} with $f=\pa_3\omega$ or $\partial_3^2 v^3$
and $g=v^3$, we achieve
\begin{equation}\begin{split}\label{nuomss2}
\Bigl|\int_{\R^3}(\pa_2v^3\pa_3v^1-\pa_1v^3\pa_3v^2)\omega_{\frac12}\,dx\Bigr|
\lesssim\Bigl(\|\pa_3\omega&\|_{L^{\frac32}}\|\pa_3 v^3\|_{\hh}^{\f16}+
\|\pa_3^2 v^3\|_{\hm}\|\nablah v^3\|_{\hh}^{\f16}\Bigr)\\
&\times\|v^3\|_{\hh}^{\f13}\|\nablah v^3\|_{\hh}^{\f12}
\bigl\|\omss\bigr\|_{L^2}^{\frac13}\bigl\|\pa_3\omss\bigr\|_{L^2}^{\frac13}.
\end{split}\end{equation}
Substituting \eqref{2.3} into the above estimate and then using
Young's inequality, we infer
\begin{equation*}\begin{split}\label{estimateomega2}
\Bigl|\int_{\R^3}(\pa_2v^3\pa_3v^1-\pa_1v^3\pa_3v^2)\omega_{\frac12}\,dx\Bigr|
&\leq \frac19\|\nablah v^3\|_{H^{\frac12,0}}^{2}
+C\Big(\|v^3\|_{\hh}^{\f49}\bigl\|\omss\bigr\|_{L^2}^{\frac89}
\|\pa_3 v^3\|_{\hh}^{\f29}\bigl\|\pa_3\omss\bigr\|_{L^2}^{\frac{16}9}\\
&\qquad\qquad\qquad\quad+\|v^3\|_{\hh}^{\f12}\bigl\|\omss\bigr\|_{L^2}^{\frac12}
\|\pa_3^2 v^3\|_{\hm}^{\f32}\bigl\|\pa_3\omss\bigr\|_{L^2}^{\frac12}\Bigr)\\
&\leq \frac19\wt N+C\bigl(M^{\f23}+M^{\f12}\bigr)N.
\end{split}\end{equation*}
Inserting the above estimate and \eqref{estimateomega1}
into \eqref{estimateomega} yields \eqref{estimateomss}. This completes the proof of this proposition.
\end{proof}

\begin{proof}[Proof of Lemma \ref{lemomhalf}]
We first get, by using H\"older's inequality and Sobolev inequality, that
\begin{equation}\begin{split}\label{lem3.1-1}
\Bigl|\int_{\R^3}\nablah\Delta_{\rm{h}}^{-1}f\cdot\nablah g
\cdot\omega_{\frac12}\,dx\Bigr|
&\leq\|\nablah\Delta_{\rm{h}}^{-1}f\|_{L^4_\h(L^2_\v)}
\|\nablah g\|_{L^{\f{12}5}_\h(L^2_\v)}
\bigl\|\omega_{\frac12}\bigr\|_{L^{3}_\h(L^\infty_\v)}\\
&\leq\|\nablah\Delta_{\rm{h}}^{-1}f\|_{H^{\f12,0}}
\|\nablah g\|_{H^{\f16,0}}
\bigl\|\omss\bigr\|_{L^{2}_\h(L^\infty_\v)}^{\f23}\\
&\leq\|f\|_{H^{-\f12,0}}\|g\|_{H^{\f12,0}}^{\f13}
\|\nablah g\|_{H^{\f12,0}}^{\f23}
\bigl\|\omss\bigr\|_{L^2}^{\f13}\bigl\|\pa_3\omss\bigr\|_{L^2}^{\f13}.
\end{split}\end{equation}
Alternatively, we can handle the estimate as follows
\begin{equation}\begin{split}\label{lem3.1-2}
\Bigl|\int_{\R^3}\nablah\Delta_{\rm{h}}^{-1}f\cdot\nablah g
\cdot\omega_{\frac12}\,dx\Bigr|
&\leq\|\nablah\Delta_{\rm{h}}^{-1}f\|_{L^6_\h(L^{\f32}_\v)}
\|\nablah g\|_{L^2_\h(L^3_\v)}
\bigl\|\omega_{\frac12}\bigr\|_{L^{3}_\h(L^\infty_\v)}\\
&\leq\|f\|_{L^{\f32}}\|\nablah g\|_{H^{0,\f16}}
\bigl\|\omss\bigr\|_{L^{2}_\h(L^\infty_\v)}^{\f23}\\
&\leq\|f\|_{L^{\f32}}\|g\|_{H^{\f12,0}}^{\f13}
\|\nablah g\|_{H^{\f12,0}}^{\f12}\|\pa_3 g\|_{H^{\f12,0}}^{\f16}
\bigl\|\omss\bigr\|_{L^2}^{\f13}\bigl\|\pa_3\omss\bigr\|_{L^2}^{\f13}.
\end{split}\end{equation}
Combining \eqref{lem3.1-1} with  \eqref{lem3.1-2} leads to  the lemma.
\end{proof}

\begin{prop}\label{propestimatepa3v3Ht}
{\sl Under the assumptions of Proposition \ref{propestimateomss},
for any $t<T^\ast$, there holds
\begin{equation}\label{estimatepa3v3Ht}
\f{d}{dt}\|\pa_3 v^3(t)\|_{\hm}^2+2\|\nabla_\nu \pa_3 v^3(t)\|_{\hm}^2
\leq \f4{9}\wt N+C\bigl(M^{\f54}+M^{\f12}\bigr)N.
\end{equation}
}\end{prop}

\begin{proof}
We first get,
by taking $\hm$ inner product
of  the $v^3$ equation of \eqref{omegav3}
with $-\p_3^2v^3,$ that
\begin{equation}\begin{split}\label{S3eq1b}
\f{d}{dt}\|\pa_3 v^3(t)\|_{\hm}^2+&2\|\nabla_\nu \pa_3 v^3(t)\|_{\hm}^2
=-\sum_{i=1}^3 \bigl(P_{i}(v,v) \, |\, \pa_3v^3\bigr)_{\hm}
\quad\mbox{with}\quad\\
 P_{1}(v,v) \eqdefa& 2\bigl({\rm Id}+\pa_3^2\Delta^{-1}\bigr)
(\pa_3 v^3)^2+2\pa_3^2\Delta^{-1}\Bigl(
\sum\limits_{\ell,m=1}^2\pa_\ell v^m\pa_m v^\ell\Bigr),\\
 P_{2}(v,v)\eqdefa& 2\bigl({\rm Id}+2\pa_3^2\Delta^{-1}\bigr)
\bigl(\pa_3 v^\h\cdot\na_\h v^3\bigr)
\quad \mbox{and}\quad P_{3}(v,v) \eqdefa 2 v\cdot \nabla\pa_3 v^3.
\end{split}\end{equation}

In what follows, we shall handle term by term the righthand side of \eqref{S3eq1b}.

\noindent$\bullet$\underline{
The estimate of  $\bigl(P_{1}(v,v) \, |\, \p_3v^3\bigr)_{\hm}$.}

Since $\pa_3^2\Delta^{-1}$ is a bounded Fourier multiplier,
we get, by using \eqref{helmdecom}, that
\begin{equation}\begin{split}\label{3.12}
\bigl|\bigl(&P_{1}(v,v) \, |\, \pa_3v^3\bigr)_{\hm}\bigr|\\
&\lesssim\|P_{1}(v,v)\|_{(B^{-1}_{2,\infty})_\h(H^{\f14})_\v}
\|\pa_3 v^3\|_{(B^{0}_{2,1})_\h(H^{-\f14})_\v}\\
&\lesssim\bigl(\|(\na_\h^2\D_\h^{-1}\omega)^2\|_{(B^{-1}_{2,\infty})_\h(H^{\f14})_\v}
+\|(\pa_3 v^3)^2\|_{(B^{-1}_{2,\infty})_\h(H^{\f14})_\v}\bigr)
\|\pa_3 v^3\|_{(B^{0}_{2,1})_\h(H^{-\f14})_\v}.
\end{split}\end{equation}
While it follows from \eqref{lemproduct3} and \eqref{2.7}
 that
\begin{equation*}\begin{split}
\|(\na_\h^2\D_\h^{-1}\omega)^2&\|_{(B^{-1}_{2,\infty})_\h(H^{\f14})_\v}
+\|(\pa_3 v^3)^2\|_{(B^{-1}_{2,\infty})_\h(H^{\f14})_\v}\\
\lesssim &\|\omega\|_{H^{0,\f38}}^2+\|\pa_3v^3\|_{H^{0,\f38}}^2\\
\lesssim &\bigl\|\omss\bigr\|_{L^2}^{\f{11}{12}}
\bigl\|\nablah\omss\bigr\|_{L^2}^{\f23}
\bigl\|\pa_3\omss\bigr\|_{L^2}^{\f{13}{12}}
+\|v^3\|_{\hh}^{\f14}\|\pa_3v^3\|_{\hh}^{\f34}
\|\pa_3^2v^3\|_{\hm}.
\end{split}\end{equation*}
And we deduce from Proposition 2.22 of \cite{BCD} that
\begin{equation*}\begin{split}
\|\pa_3 v^3\|_{(B^{0}_{2,1})_\h(H^{-\f14})_\v}
&\leq\|\pa_3v^3\|_{H^{-\f14,-\f14}}^{\f12}\|\pa_3v^3\|_{H^{\f14,-\f14}}^{\f12}\\
&\leq\bigl(\|v^3\|_{\hh}^{\f14}\|\pa_3v^3\|_{\hm}^{\f34}\bigr)^{\f12}
\bigl(\|v^3\|_{\hh}^{\f14}\|\pa_3v^3\|_{\hm}^{\f14}
\|\pa_3v^3\|_{\hh}^{\f12}\bigr)^{\f12}\\
&=\|v^3\|_{\hh}^{\f14}\|\pa_3v^3\|_{\hm}^{\f12}
\|\pa_3v^3\|_{\hh}^{\f14}.
\end{split}\end{equation*}
Substituting the above inequalities into \eqref{3.12} gives rise to
$$\bigl|\bigl(P_{1}(v,v) \, |\,\pa_3v^3\bigr)_{\hm}\bigr|
\lesssim M^{\f38}N^{\f18}\bigl(M^{\f{11}{24}}\wt N^{\f13}N^{\f{13}{24}}
+M^{\f18}N^{\f78}\bigr).$$
Applying Young's inequality yields
\begin{equation}\label{3.13}
\bigl|\bigl(P_1(v,v) \, |\, \pa_3v^3\bigr)_{\hm}\bigr|
\leq\frac1{9}\wt N+C\bigl(M^{\f54}+M^{\f12}\bigr)N.
\end{equation}

\noindent$\bullet$\underline{
The estimate of  $\bigl(P_{2}(v,v) \, |\, \p_3v^3\bigr)_{\hm}$.}

By using integration by parts, we can write
$\bigl(P_{2}(v,v) \, |\, \pa_3v^3\bigr)_{\hm}=\cA_{1}+\cA_{2}$ with
\beq\begin{split}\label{3.14}
\cA_{1}\eqdefa &-2\bigl(({\rm Id}+2\pa_3^2\Delta^{-1})
(v^\h\cdot\na_\h v^3)
\, |\, \pa_3^2 v^3\bigr)_{\hm},\\
\cA_{2}\eqdefa&-2\bigl(({\rm Id}+2\pa_3^2\Delta^{-1})
( v^\h\cdot\na_\h\pa_3 v^3)
\, |\,\pa_3v^3\bigr)_{\hm}.
\end{split}\eeq
Applying the law of product, Lemma \ref{lemproductlaw}, and Proposition 2.22 of \cite{BCD}  yields
\begin{equation}\begin{split}\label{3.15}
|\cA_{1}|&\lesssim\|v^\h\cdot\na_\h v^3\|_{H^{-\f34,\f14}}
\bigl\|\pa_3^2 v^3\bigr\|_{H^{-\f14,-\f14}}\\
&\lesssim\|v^\h\|_{H^{\f34,\f14}}
\|\na_\h v^3\|_{(H^{-\frac12})_\h(B^{\f12}_{2,1})_\v}
\|\pa_3 v^3\|_{\hh}^{\f14}\|\pa_3^2v^3\|_{\hm}^{\f34}\\
&\lesssim\|v^\h\|_{H^{\f34,\f14}}\|v^3\|_{\hh}^{\f12}
\|\pa_3 v^3\|_{\hh}^{\f34}\|\pa_3^2v^3\|_{\hm}^{\f34}.
\end{split}\end{equation}
While we deduce from \eqref{helmdecom} and Proposition  \ref{lemomega} that
\begin{align*}
\|v^\h\|_{H^{\f34,\f14}}&\lesssim
\|\omega\|_{H^{-\f14,\f14}}+\|\pa_3v^3\|_{H^{-\f14,\f14}}\\
&\lesssim\bigl\|\omss\bigr\|_{L^2}^{\f56}
\bigl\|\nablah \omss\bigr\|_{L^2}^{\f1{12}}
\bigl\|\pa_3 \omss\bigr\|_{L^2}^{\f5{12}}
+\|\pa_3v^3\|_{\hm}^{\f12}
\|\pa_3 v^3\|_{\hh}^{\f14}\|\pa_3^2v^3\|_{\hm}^{\f14}.
\end{align*}
Inserting the above inequalities  into \eqref{3.15} and then using Young's inequality,
we obtain
\begin{equation}\label{3.16}
\begin{split}
|\cA_{1}|
\leq& C\bigl(M^{\f23}\wt N^{\f1{24}}N^{\f{23}{24}}+M^{\f12}N\bigr)\\
\leq&\f1{9}\wt N+C\bigl(M^{\f{16}{23}}+M^{\f12}\bigr)N.
\end{split}
\end{equation}
Along the same line, we have
\begin{equation}\begin{split}\label{3.17}
|\cA_{2}|&\lesssim\|v^\h\cdot\na_\h \pa_3v^3\|
_{(H^{-\f34})_\h(B^{-\f12}_{2,\infty})_\v}\cdot
\bigl\|\pa_3 v^3\bigr\|_{(H^{-\f14})_\h(B^{\f12}_{2,1})_\v}\\
&\lesssim\|v^\h\|_{H^{\f34,0}}
\|\na_\h\pa_3 v^3\|_{H^{-\frac12,0}}\cdot
\|\pa_3 v^3\|_{H^{-\f14,\f14}}^{\f12}\|\pa_3v^3\|_{H^{-\f14,\f34}}^{\f12}\\
&\lesssim\bigl(\bigl\|\omss\bigr\|_{L^2}^{\f{13}{12}}
\bigl\|\nablah \omss\bigr\|_{L^2}^{\f1{12}}
\bigl\|\pa_3 \omss\bigr\|_{L^2}^{\f16}
+\|\pa_3v^3\|_{\hm}^{\f34}\|\pa_3 v^3\|_{\hh}^{\f14}\bigr)
\|\pa_3v^3\|_{\hh}\\
&\qquad\times\bigl(\|v^3\|_{\hh}^{\f14}\|\pa_3v^3\|_{\hm}^{\f14}
\|\pa_3^2v^3\|_{\hm}^{\f12}\bigr)^{\f12}
\bigl(\|\pa_3 v^3\|_{\hh}^{\f14}\|\pa_3^2v^3\|_{\hm}^{\f34}\bigr)^{\f12}.
\end{split}\end{equation}
Then applying Young's inequality yields
\begin{equation}\label{3.18}
\begin{split}
|\cA_{2}|\leq &C\bigl(M^{\f23}\wt N^{\f1{24}}N^{\f{23}{24}}+M^{\f12}N\bigr)\\
\leq&\f1{9}\wt N+C\bigl(M^{\f{16}{23}}+M^{\f12}\bigr)N,
\end{split}
\end{equation}
which together with  \eqref{3.16} ensures that
\begin{equation}\label{3.19}
\bigl|\bigl(P_{2}(v,v) \, |\, \pa_3v^3\bigr)_{\hm}\bigr|
\leq\f2{9}\wt N+C\bigl(M^{\f{16}{23}}+M^{\f12}\bigr)N.
\end{equation}

\noindent$\bullet$\underline{
The estimate of  $\bigl(P_{3}(v,v) \, |\, \p_3v^3\bigr)_{\hm}$.}

It is easy to observe that the term
$\bigl|\bigl(\vh\cdot\nablah\pa_3 v^3 \, |\, \pa_3v^3\bigr)_{\hm}\bigr|$
shares the same estimate as $\cA_2$ given by \eqref{3.17}.
It remains to handle the estimate of $\bigl|\bigl(v^3 \pa_3^2 v^3 \, |\, \pa_3v^3\bigr)_{\hm}\bigr|.$
Indeed by using  the law of product, Lemma \ref{lemproductlaw}, we get
\begin{equation*}\begin{split}
\bigl|\bigl(v^3 \pa_3^2 v^3 \, |\, \pa_3v^3\bigr)_{\hm}\bigr|
&\lesssim\bigl\|v^3\p_3^2v^3\bigr\|_{H^{-\f34,-\f14}}
\cdot\|\partial_3v^3\|_{H^{-\f14,\f14}}\\
&\lesssim\|v^3\|_{(H^{\frac12})_\h(B^{\f12}_{2,1})_\v}\cdot
\|\pa_3^2v^3\|_{H^{-\f14,-\f14}}\cdot\|v^3\|_{\hh}^{\f14}
\|\pa_3v^3\|_{\hm}^{\f14}\|\pa_3^2v^3\|_{\hm}^{\f12}\\
&\lesssim\|v^3\|_{\hh}^{\f12}\|\pa_3v^3\|_{\hh}^{\f12}\cdot
\|v^3\|_{\hh}^{\f14}\|\pa_3v^3\|_{\hm}^{\f14}
\|\pa_3v^3\|_{\hh}^{\f14}\|\pa_3^2v^3\|_{\hm}^{\f54}\\
&\lesssim\|v^3\|_{\hh}^{\f34}\|\pa_3v^3\|_{\hm}^{\f14}
\|\pa_3v^3\|_{\hh}^{\f34}\|\pa_3^2v^3\|_{\hm}^{\f54}\\
&\lesssim M^{\f12}N.
\end{split}\end{equation*}
By combining the above inequality with \eqref{3.18},  we deduce that
\beq\begin{split}\label{3.20}
\bigl|\bigl(P_{3}(v,v) \, |\,\pa_3v^3\bigr)_{\Ht}\bigr|
\leq\f1{9}\wt N+C\bigl(M^{\f{16}{23}}+M^{\f12}\bigr)N.
\end{split}
\eeq

Inserting  \eqref{3.13},~\eqref{3.19} and \eqref{3.20} into \eqref{S3eq1b} leads to \eqref{estimatepa3v3Ht}.
This completes the proof of Proposition \ref{propestimatepa3v3Ht}.
\end{proof}

\begin{prop}\label{propestimatepa3v3H120}
{\sl Under the assumptions of Proposition \ref{propestimateomss},
for any $t<T^\ast$, there holds
\begin{equation}\label{estimatepa3v3H120}
\begin{split}
\f{d}{dt}\|v^3&(t)\|_{H^{\frac12,0}}^2+2\|\nabla_\nu v^3(t)\|_{H^{\frac12,0}}^2
\leq \f29\wt N+C\bigl(M^2+M^{\f12}\bigr)N.
\end{split}\end{equation}
}\end{prop}
\begin{proof}
By taking the $H^{\frac12,0}$ inner product of the $v^3$ equation of \eqref{omegav3} with $v^3$, we write
\begin{equation}\label{estimatev3}
\begin{split}
\f{d}{dt}\|v^3(t)\|_{H^{\frac12,0}}^2+&2\|\nabla_\nu v^3(t)\|_{H^{\frac12,0}}^2
=-\sum_{i=1}^3 \bigl(Q_i(v,v) \, |\, v^3\bigr)_{H^{\frac12,0}}
\quad\mbox{with}\quad\\
 Q_1(v,v) \eqdefa& 2\partial_3\D^{-1} (\pa_3 v^3)^2+
  2\partial_3\D^{-1}\Bigl(\sum_{\ell,m=1}^2\partial_\ell v^m\partial_m  v^\ell\Bigr),\\
Q_2(v,v)\eqdefa&  4\partial_3 \D^{-1}\Bigl(\sum_{\ell=1}^2 \partial_3 v^\ell
 \partial_\ell v^3\Bigr)
\quad \mbox{and}\quad Q_3(v,v) \eqdefa 2v\cdot \nabla v^3.
 \end{split}
\end{equation}

\noindent$\bullet$\underline{
The estimate of  $\bigl(Q_{1}(v,v) \, |\, v^3\bigr)_{H^{\frac12,0}}$.}

In view of \eqref{helmdecom}, we get, by applying the law of product,
Lemma \ref{lemproductlaw},  and Proposition \ref{lemomega}, that
\begin{equation*}\begin{split}\label{3.23}
\bigl|\bigl(Q_{1}(v,v) \,|\, v^3\bigr)_{H^{\frac12,0}}\bigr|
&\leq\|Q_{1}(v,v)\|_{\hh}\|v^3\|_{\hh}\\
&\lesssim\bigl(\bigl\|(\pa_3 v^3)^2\bigr\|_{H^{-\f23,\f16}}
+\bigl\||\nablah\vh|^2\bigr\|_{H^{-\f23,\f16}}\bigr)\|v^3\|_{\hh}\\
&\lesssim\bigl(\|\pa_3 v^3\|_{H^{\f16,\f13}}^2
+\|\omega\|_{H^{\f16,\f13}}^2\bigr)
\|v^3\|_{\hh}\\
&\lesssim\bigl(\|\pa_3 v^3\|_{\hh}^{\f43}\|\pa_3^2v^3\|_{\hm}^{\f23}
+\bigl\|\omss\bigr\|_{L^2}^{\f23}
\bigl\|\nablah \omss\bigr\|_{L^2}
\bigl\|\pa_3 \omss\bigr\|_{L^2}\bigr)\|v^3\|_{\hh}.
\end{split}\end{equation*}
By applying Young's inequality, we achieve
\begin{equation}\label{3.24}
\bigl|\bigl(Q_{1}(v,v) \,|\,v^3\bigr)_{H^{\frac12,0}}\bigr|
\leq \f19\bigl\|\nablah\omss\bigr\|_{L^2}^2
+C\bigl(M^{\f53}+M^{\f12}\bigr)N.
\end{equation}

\noindent$\bullet$\underline{The estimate of $\bigl(Q_{2}(v,v) \,|\,v^3\bigr)_{H^{\frac12,0}}.$}

By using integration by parts, we write
\begin{equation}\begin{split}\label{S4splitQ2}
&\bigl(Q_{2}(v,v) \,|\,v^3\bigr)_{H^{\frac12,0}}=-\cB_{1}-\cB_{2},
\quad\mbox{where}\\
\cB_{1}\eqdefa 4\bigl(\pa_3\Delta^{-1}(v^\h& \cdot\na_\h v^3)
\,\big|\,\pa_3v^3\bigl)_{H^{\frac12,0}},\quad
\cB_{2}\eqdefa 4\bigl(\pa_3\Delta^{-1}( v^\h\cdot\na_\h\p_3 v^3)
\,\big|\,v^3\bigl)_{H^{\frac12,0}}.
\end{split}\end{equation}
We first get, by applying the law of product, Lemma \ref{lemproductlaw}, that
\begin{align*}
|\cB_{1}|&\lesssim \bigl\|\pa_3\Delta^{-1}(v^\h\cdot\na_\h v^3) \bigr\|_{H^{\frac12,0}}\|\pa_3v^3\|_{H^{\frac12,0}}\\
&\lesssim \|v^\h\cdot\na_\h v^3\|_{H^{-\frac56,\f13}}\|\pa_3v^3\|_{H^{\frac12,0}}\\
&\lesssim \|v^\h\|_{H^{\f23,\f13}}
\|\na_\h v^3\|_{(H^{-\f12})_\h(B^{\f12}_{2,1})_{\v}}
\|\p_3 v^3\|_{H^{\frac12,0}}.
\end{align*}
Yet by virtue of \eqref{helmdecom}, we deduce from
Proposition \ref{lemomega} that
\beq \label{ping1}
\begin{split}
\|v^\h\|_{H^{\f23,\f13}}
&\lesssim\|\omega\|_{H^{-\f13,\f13}}+\|\pa_3v^3\|_{H^{-\f13,\f13}}\\
&\lesssim\bigl\|\omss\bigr\|_{L^2}^{\f56}
\bigl\|\pa_3 \omss\bigr\|_{L^2}^{\f12}
+\|v^3\|_{\hh}^{\f16}\|\pa_3v^3\|_\hm^{\f13}
\|\pa_3^2v^3\|_\hm^{\f12}.
\end{split}
\eeq
And it follows from Proposition 2.22 of \cite{BCD} that
$$\|\na_\h v^3\|_{(H^{-\f12})_\h(B^{\f12}_{2,1})_{\v}}
\leq\|v^3\|_{\hh}^{\f12}\|\pa_3v^3\|_\hh^{\f12}.$$
As a result, it comes out
\begin{equation}\begin{split}\label{3.26}
|\cB_{1}|\lesssim \bigl(\bigl\|\omss\bigr\|_{L^2}^{\f56}
\bigl\|\pa_3 \omss\bigr\|_{L^2}^{\f12}
+\|v^3\|_{\hh}^{\f16}\|\pa_3v^3\|_\hm^{\f13}
\|\pa_3^2v^3\|_\hm^{\f12}\bigr)
\|v^3\|_{\hh}^{\f12}\|\pa_3v^3\|_\hh^{\f32}.
\end{split}\end{equation}

Along the same line, we deduce that
\beno
\begin{split}
|\cB_{2}|&\lesssim\bigl\|\D^{-1}(v^\h\cdot\na_\h\pa_3v^3)\bigr\|_{H^{\f12,\f12}}
\|\pa_3v^3\|_{H^{\f12,-\f12}}\\
& \lesssim\|v^\h\cdot\na_\h\pa_3v^3\|_{H^{-\f56,-\f16}}
\|v^3\|_{\hh}^{\f12}\|\pa_3v^3\|_{\hh}^{\f12}\\
& \lesssim\|v^\h\|_{H^{\f23,\f13}}\|\na_\h\pa_3v^3\|_{H^{-\f12,0}}
\|v^3\|_{\hh}^{\f12}\|\pa_3v^3\|_{\hh}^{\f12}\\
& \lesssim\bigl(\bigl\|\omss\bigr\|_{L^2}^{\f56}
\bigl\|\pa_3 \omss\bigr\|_{L^2}^{\f12}
+\|v^3\|_{\hh}^{\f16}\|\pa_3v^3\|_\hm^{\f13}
\|\pa_3^2v^3\|_\hm^{\f12}\bigr)
\|v^3\|_{\hh}^{\f12}\|\pa_3v^3\|_{\hh}^{\f32},
\end{split}
\eeno
from which and  \eqref{3.26}, we infer
\begin{equation}\begin{split}\label{3.28}
\bigl|\bigl(&Q_{2}(v,v) \,|\,v^3\bigr)_{H^{\frac12,0}}\bigr|\\
& \lesssim\bigl(\bigl\|\omss\bigr\|_{L^2}^{\f56}
\bigl\|\pa_3 \omss\bigr\|_{L^2}^{\f12}
+\|v^3\|_{\hh}^{\f16}\|\pa_3v^3\|_\hm^{\f13}
\|\pa_3^2v^3\|_\hm^{\f12}\bigr)
\|v^3\|_{\hh}^{\f12}\|\pa_3v^3\|_{\hh}^{\f32}\\
&\leq C\bigl(M^{\f23}+M^{\f12}\bigr)N.
\end{split}\end{equation}

\noindent$\bullet$\underline{The  estimate  of $\bigl(Q_{3}(v,v) \,|\,v^3\bigr)_{H^{\frac12,0}}.$}

We deduce from Lemma \ref{lemproductlaw} and  Proposition \ref{lemomega} that
\beno
\begin{split}
\bigl|\bigl(\vh\cdot\nablah v^3 \,|\,v^3\bigr)_{H^{\frac12,0}}\bigr|
&\lesssim \|\vh\cdot\nablah v^3\|_{H^{-\f12,0}}
\|v^3\|_{H^{\frac32,0}}\\
&\lesssim \|\vh\|_{H^{\f23,\f13}}\|\nablah v^3\|_{H^{-\f16,\f16}}
\|\nablah v^3\|_{H^{\frac12,0}}.
\end{split}
\eeno
Inserting \eqref{ping1} into the above inequality yields
\beno
\begin{split}
\bigl|\bigl(\vh\cdot\nablah v^3 \,|\,v^3\bigr)_{H^{\frac12,0}}\bigr|
 \lesssim \bigl(\bigl\|\omss\bigr\|_{L^2}^{\f56}
\bigl\|\pa_3 \omss\bigr\|_{L^2}^{\f12}
&+\|v^3\|_{\hh}^{\f16}\|\pa_3v^3\|_\hm^{\f13}
\|\pa_3^2v^3\|_\hm^{\f12}\bigr)\\
&\qquad\quad\times
\|v^3\|_{\hh}^{\f12}\|\nablah v^3\|_{\hh}^{\f43}\|\pa_3v^3\|_{\hh}^{\f16}.
\end{split}
\eeno
While it follows from the law of product,  Lemma \ref{lemproductlaw}, that
\begin{align*}
\bigl|\bigl(v^3\pa_3 v^3 \,|\,v^3\bigr)_{H^{\frac12,0}}\bigr|
&\lesssim \|v^3\pa_3v^3\|_{H^{0,-\f14}}\|v^3\|_{H^{1,\f14}}\\
&\lesssim\|v^3\|_{H^{\f12,\f14}}\|\pa_3v^3\|_{\hh}\|v^3\|_{H^{1,\f14}}\\
&\lesssim\|v^3\|_\hh\|\nablah v^3\|_\hh^{\f12}\|\pa_3 v^3\|_\hh^{\f32}.
\end{align*}
By combining the above two inequalities, we obtain
\begin{equation*}\begin{split}\label{3.29}
\bigl|\bigl(Q_{3} (v,v)\,|\,v^3\bigr)_{H^{\frac12,0}}\bigr|
\lesssim&\bigl(\bigl\|\omss\bigr\|_{L^2}^{\f56}
\bigl\|\pa_3 \omss\bigr\|_{L^2}^{\f12}
+\|v^3\|_{\hh}^{\f16}\|\pa_3v^3\|_\hm^{\f13}
\|\pa_3^2v^3\|_\hm^{\f12}\bigr)\\
&\times\|v^3\|_{\hh}^{\f12}\|\nablah v^3\|_{\hh}^{\f43}\|\pa_3v^3\|_{\hh}^{\f16}
+\|v^3\|_\hh\|\nablah v^3\|_\hh^{\f12}\|\pa_3 v^3\|_\hh^{\f32}.
\end{split}\end{equation*}
Applying Young's inequality yields
\begin{equation}\begin{split}\label{3.30}
\bigl|\bigl(Q_{3} (v,v)\,|&\,v^3\bigr)_{H^{\frac12,0}}\bigr|
\leq \f19\|\nablah v^3\|_\hh^2+C\bigl(M^2+M^{\f23}\bigr)N.
\end{split}\end{equation}

By inserting the estimates \eqref{3.24},~\eqref{3.28}
and \eqref{3.30} into \eqref{estimatev3}, we achieve
 \eqref{estimatepa3v3H120}. This completes the proof of this proposition.
\end{proof}

\subsection{The proof of Theorem \ref{thmmain}}

With the {\it a priori} estimates obtained in
the previous subsection, we are now in a position to
complete the proof of Theorem \ref{thmmain}.

\begin{proof}[Proof of Theorem \ref{thmmain}]
We first deduce from
Theorem \ref{thm1.4ofCZ} that
 under the assumptions of Theorem \ref{thmmain}, $(NS_\nu)$ has a unique solution
$v\in \cE_{T^\ast}$~(see \eqref{defcE} for the definition),
where $T^\ast$ is the maximal time of  existence.

In what follows, we shall prove that  \eqref{thmmainestimate}
holds for any $t\in[0,T^\ast[.$
To do this,  we define
\begin{equation}\label{defTstar}\begin{split}
T^\star \eqdefa\sup\Bigl\{\, T\in ]0,T^\ast[\,:\,\mbox{so that \eqref{thmmainestimate} holds for any $t\in[0,T]$}
\, \Bigr\}.
\end{split}\end{equation}
We are going to prove that $T^\star=T^\ast.$ Otherwise,
if $T^\star<T^\ast$,  for any $t\in[0,T^\star[$, we get, by summing up
\eqref{estimateomss},~\eqref{estimatepa3v3Ht}, and \eqref{estimatepa3v3H120}, that
\begin{equation*}
\frac{d}{dt}M(t)+\wt N(t)+\nu^2N(t)\\
\leq C\bigl(M^2(t)+M^{\f12}(t)\bigr)N(t).
\end{equation*}
Thanks to \eqref{defTstar}, for any $t< T^\star,$ by
integrating the above inequality over $[0,t],$ we obtain
\begin{equation}\label{S5eq23}
M(t)+\int_0^t\bigl(\wt N(t')+\nu^2 N(t')\bigr)\,dt'
\leq M_0+C \bigl(M_0^2+M_0^{\f12}\bigr)\int_0^t N(t')\,dt'.
\end{equation}
On the other hand, if $C_1$ in \eqref{thmmaincondi} is so large
that $C_1\geq 2C$, we find
\begin{equation*}\label{continuitylargecondi}
\frac14\nu^2\geq C \bigl(M_0^2+M_0^{\frac12}\bigr).
\end{equation*}
Then we deduce from \eqref{S5eq23} that
$$M(t)+\frac34\int_0^t\bigl(\wt N(t')+\nu^2 N(t')\bigr)\, dt'
\leq M_0.$$
This in particular implies that for any $t$ in $[0,T^\star[$,
there holds
$$\bigl\|\omss(t)\bigr\|_{L^2}^2+\|\nabla v^3(t)\|_{\hm}^2
+\int_0^t\left(\bigl\|\nabla_\nu\omss\bigr\|_{L^2}^2
+\bigl\|\nabla_\nu \nabla v^3\bigr\|_{\hm}^2\right)\, dt'
\leq \frac43 M_0,$$
which contracts with the definition of $T^\star$ given by \eqref{defTstar}.
This in turn shows that $T^\star=T^\ast$,
and \eqref{thmmainestimate} holds for any $t\in[0,T^\ast[$.

Now if $T^\ast<\infty$, by virtue of  \eqref{thmmainestimate} and \eqref{v3anisotropic},
 for any $p\in]4,6[,$ we find
\beq\label{star}
\begin{split}
\int_0^{T^\ast}\|v^3(t)\|_{H^{\frac12+\frac2p}}^p\,dt
\lesssim&\sup_{t\in [0,T^\ast[}\Bigl( \|\pa_3 v^3(t)\|_{\hm}^{\frac{p-4}{2}}\|v^3(t)\|_{\hh}^{\f{p}{2}}\Bigr)
\int_0^{T^\ast}\|\pa_3^2 v^3(t)\|_{\hm}^{2}\,dt\\
&+\sup_{t\in [0,T^\ast[}\|v^3(t)\|_{H^{\frac12,0}}^{p-2}\int_0^{T^\ast}\|\nabla_\h v^3(t)\|_{H^{\frac12,0}}^2\,dt\\
\lesssim& \bigl(1+\nu^{-2}\bigr)M_0^{\f{p}{2}},
\end{split}
\eeq
which contradicts with the blow-up criterion \eqref{blowupCZ5}.
Hence $T^\ast=\infty$, and the estimate
\eqref{thmmainestimate} holds for any $t\in[0,\infty[$.
 This completes the proof of Theorem \ref{thmmain}.
\end{proof}

\setcounter{equation}{0}
\section{Global solution of \eqref{S1eq1} with a fast variable}

In this section, we shall adapt the arguments in the previous section to prove Theorem \ref{thmnu}.

\subsection{{\it A priori} estimates} Taking $L^{2}$ scalar product
of $\omega$ equation in \eqref{omegav3nu} with $\om_{\f12}$ gives
\begin{equation}\label{4.1}
\frac{d}{dt}\bigl\|\omss\bigr\|_{L^2}^2+
\frac43 \bigl\|\nabla_\nu\omss\bigr\|_{L^2}^2
=\frac32\nu\int_{\R^3}\pa_3v^3|\omega|^{\frac32}
-(\pa_3\vh\cdot\nablah^\perp v^3)\omega_{\frac12}\,dx.
\end{equation}
We point out that the main difficulty on the estimate of the right-hand side of \eqref{4.1} lies in the fact that there appears $\nu.$
We first get, by using integration by parts, that
\begin{equation}\begin{split}\label{4.2}
\nu\Bigl|\int_{\R^3}\pa_3v^3|\omega|^{\frac32}\,dx\Bigr|
&\lesssim\nu\int_{\R^3}|v^3||\pa_3\omega||\omega|^{\frac12}\,dx\\
&\lesssim\nu\|v^3\|_{L^6}\|\pa_3\omega\|_{L^{\frac32}}
\|\om\|_{L^3}^{\frac12}.
\end{split}\end{equation}
Yet it follows from Sobolev inequality and  Proposition \ref{lemomega} that
\begin{align*}
\|v^3\|_{L^6}\|\om\|_{L^3}^{\frac12}
&\lesssim\|v^3\|_{H^{\f23,\f13}}\|\om\|_{H^{\f13,\f16}}^{\frac12}\\
&\lesssim\|v^3\|_{\hs}^{\frac12}
\|\pa_3 v^3\|_{\hs}^{\frac12}
\bigl\|\omss\bigr\|_{L^2}^{\frac16}
\bigl\|\nablah\omss\bigr\|_{L^2}^{\frac13}
\bigl\|\pa_3\omss\bigr\|_{L^2}^{\frac16}.
\end{align*}
Inserting the above inequality and  \eqref{2.3} into
 \eqref{4.2} gives rise to
\begin{equation}\begin{split}\label{4.3}
\nu\Bigl|\int_{\R^3}\pa_3v^3|\omega|^{\frac32}\,dx\Bigr|
\lesssim\nu^{-\f23}\| v^3\|_{\hs}^{\frac12}
\|\nu\pa_3 v^3\|_{\hs}^{\frac12}
\bigl\|\omss\bigr\|_{L^2}^{\frac12}
\bigl\|\nablah\omss\bigr\|_{L^2}^{\frac13}
\bigl\|\nu\pa_3\omss\bigr\|_{L^2}^{\frac76}.
\end{split}\end{equation}
While by first applying \eqref{lem3.1-2} with $f=\pa_3\omega$ or $\partial_3^2 v^3$
and $g=v^3$, and then \eqref{2.3} and Proposition \ref{lemomega}, we achieve
\begin{equation*}\begin{split}
\nu\Bigl|\int_{\R^3}(\pa_3\vh\cdot\nablah^\perp v^3)\omega_{\frac12}\,dx\Bigr|
&\lesssim\nu \bigl(\|\pa_3\omega\|_{L^{\f32}}+\|\pa_3^2 v^3\|_{L^{\f32}}\bigr)
\|\nablah v^3\|_{H^{0,\f16}}
\bigl\|\omss\bigr\|_{L^2}^{\frac13}\bigl\|\pa_3\omss\bigr\|_{L^2}^{\frac13}\\
&\lesssim\nu^{-\f23}
\Bigl(\bigl\|\omss\bigr\|_{L^2}^{\frac13}\bigl\|\nu\pa_3\omss\bigr\|_{L^2}
+\bigl\|\vss\bigr\|_{L^2}^{\frac13}\bigl\|\nu\pa_3\vss\bigr\|_{L^2}\Bigr)\\
&\quad\times\|v^3\|_{\hs}^{\f13}\|\nablah v^3\|_{\hs}^{\f13}
\|\nu\pa_3 v^3\|_{\hs}^{\f13}
\bigl\|\omss\bigr\|_{L^2}^{\frac13}\bigl\|\nu\pa_3\omss\bigr\|_{L^2}^{\frac13}.
\end{split}\end{equation*}
By inserting the above inequality and \eqref{4.3} into \eqref{4.1}, we find
\begin{equation}\begin{split}\label{4.4}
\frac{d}{dt}\bigl\|\omss\bigr\|_{L^2}^2+
\frac43 \bigl\|\nabla_\nu\omss&\bigr\|_{L^2}^2
\lesssim\nu^{-\f23}\| v^3\|_{\hs}^{\frac12}
\|\nabla_\nu v^3\|_{\hs}^{\frac12}
\bigl\|\omss\bigr\|_{L^2}^{\frac12}
\bigl\|\nabla_\nu\omss\bigr\|_{L^2}^{\frac32}\\
&+\nu^{-\f23}
\Bigl(\bigl\|\omss\bigr\|_{L^2}^{\frac13}\bigl\|\nabla_\nu\omss\bigr\|_{L^2}
+\bigl\|\vss\bigr\|_{L^2}^{\frac13}\bigl\|\nabla_\nu\vss\bigr\|_{L^2}\Bigr)\\
&\qquad\qquad\qquad\times\|v^3\|_{\hs}^{\f13}\|\nabla_\nu v^3\|_{\hs}^{\f23}
\bigl\|\omss\bigr\|_{L^2}^{\frac13}\bigl\|\nabla_\nu\omss\bigr\|_{L^2}^{\frac13}.
\end{split}\end{equation}

Next, taking $L^{2}$ scalar product
of vertical derivative to the $v^3$ equation in \eqref{omegav3nu} with
$(\p_3v^3)_{\f12}$ gives
\begin{equation}\begin{split}\label{4.5}
&\frac{d}{dt}\bigl\|\vss\bigr\|_{L^2}^2+
\frac43 \bigl\|\nabla_\nu\vss\bigr\|_{L^2}^2
=-\f32\sum_{i=1}^2\int_{\R^3} P_{i,\nu}(v,v)
| (\p_3v^3)_{\f12}\,dx
\quad\mbox{with}\quad\\
&\ P_{1,\nu}(v,v) \eqdefa 2\nu\bigl(\rm{Id}+(\nu\pa_3)^2\Delta_\nu^{-1}\bigr)
(\pa_3 v^3)^2+2\nu(\nu\pa_3)^2\Delta_\nu^{-1}
\Bigl(\sum\limits_{\ell,m=1}^2\pa_\ell v^m\pa_m v^\ell\Bigr),\\
&\qquad\qquad\quad P_{2,\nu}(v,v)\eqdefa 2\bigl(\rm{Id}
+2(\nu\pa_3)^2\Delta_\nu^{-1}\bigr)\bigl(\nu\pa_3 v^\h\cdot\na_\h v^3\bigr).
\end{split}\end{equation}
Since $(\nu\pa_3)^2\Delta_\nu^{-1}$
is a Fourier multiplier, which is bounded in any $L^p$ for $p\in ]1,\infty[$ (see Theorem 0.2.6 of \cite{Sogge} for instance),
 we can deduce from  \eqref{helmdecom} that
\begin{equation}\begin{split}\label{4.6}
\Bigl|\int_{\R^3} P_{1,\nu}(v,v) | (\p_3v^3)_{\f12}\,dx\Bigr|
&\lesssim\nu\Bigl(\bigl\|(\pa_3 v^3)^2\bigr\|_{L^{\f32}}
+\bigl\||\nablah\vh|^2\bigr\|_{L^{\f32}}\Bigr)
\bigl\|(\p_3v^3)_{\f12}\bigr\|_{L^3}\\
&\lesssim\nu\bigl(\|\pa_3 v^3\|_{L^3}^2
+\|\omega\|_{L^3}^2\bigr)
\bigl\|\vss\bigr\|_{L^2}^{\f23}.
\end{split}\end{equation}
It follows from Sobolev inequality that
\begin{align*}
\|\pa_3 v^3\|_{L^3}^2
+\|\omega\|_{L^3}^2
&\lesssim\|\pa_3 v^3\|_{H^{\f13,\f16}}^2
+\|\omega\|_{H^{\f13,\f16}}^2\\
&\lesssim\|\pa_3 v^3\|_{H^{-\f13,\f56}}^{\f23}
\|\pa_3v^3\|_{H^{\f23,-\f16}}^{\f43}
+\|\omega\|_{H^{\f13,\f16}}^2.
\end{align*}
So that applying Proposition \ref{lemomega} leads to
\begin{equation*}\begin{split}
\|\pa_3 v^3\|_{L^3}^2
+\|\omega\|_{L^3}^2
\lesssim\bigl\|\vss\bigr\|_{L^2}^{\f29}
\bigl\|\pa_3 \vss\bigr\|_{L^2}^{\f23}
\|\pa_3v^3\|_{H^{\f23,-\f16}}^{\f43}
+\bigl\|\omss\bigr\|_{L^2}^{\f23}
\bigl\|\nablah \omss\bigr\|_{L^2}^{\f43}
\bigl\|\pa_3 \omss\bigr\|_{L^2}^{\f23}.
\end{split}\end{equation*}
Substituting the above inequality into \eqref{4.6}, we achieve
\begin{equation}\begin{split}\label{4.7}
\Bigl|\int_{\R^3} P_{1,\nu}(v,v) | (\p_3v^3)_{\f12}\,dx\Bigr|
\lesssim\Bigl(&\nu^{-1}\bigl\|\vss\bigr\|_{L^2}^{\f29}
\bigl\|\nu\pa_3 \vss\bigr\|_{L^2}^{\f23}
\|\nu\pa_3v^3\|_{H^{\f23,-\f16}}^{\f43}\\
&+\nu^{\f13}\bigl\|\omss\bigr\|_{L^2}^{\f23}
\bigl\|\nablah \omss\bigr\|_{L^2}^{\f43}
\bigl\|\nu\pa_3 \omss\bigr\|_{L^2}^{\f23}\Bigr)
\bigl\|\vss\bigr\|_{L^2}^{\f23}.
\end{split}\end{equation}
To handle the second term on the right-hand side of \eqref{4.5}, we observe from Theorem 2.40 of \cite{BCD} that
\beno
\begin{split}
\Bigl|\int_{\R^3} P_{2,\nu}(v,v) | (\p_3v^3)_{\f12}\,dx\Bigr|
&\lesssim\nu\||D_3|^{-\f16}\bigl(\pa_3 \vh\cdot\nablah v^3\bigr)\|_{L^{\f32}}
\bigl\||D_3|^{\f16}(\p_3v^3)_{\f12}\bigr\|_{L^3}\\
&\lesssim\nu\|\pa_3 \vh\cdot\nablah v^3\|_{L^{\f32}_\h\bigl(L^{\f65}_\v\bigr)}\bigl\|(\p_3v^3)_{\f12}\bigr\|_{L^3_\h\bigl(B^{\f16}_{3,2}\bigr)_\v}.
\end{split}
\eeno
But we deduce from Lemma 5.1 of \cite{CZ} that
\beno
\begin{split}
\bigl\|(\p_3v^3)_{\f12}(x_\h,\cdot)\bigr\|_{\bigl(B^{\f16}_{3,2}\bigr)_\v}\lesssim &\bigl\|(\p_3v^3)_{\f12}(x_\h,\cdot)\bigr\|_{\bigl(B^{\f16}_{3,\f32}\bigr)_\v}\\
\lesssim &\bigl\|(\p_3v^3)_{\f34}(x_\h,\cdot)\bigr\|_{\bigl(B^{\f14}_{2,1}\bigr)_\v}\\
\lesssim &\bigl\|\vss(x_\h,\cdot)\bigr\|_{L^2}^{\f12}\bigl\|\pa_3\vss(x_\h,\cdot)\bigr\|_{L^2}^{\f16},
\end{split}
\eeno
Taking $L^3_\h$ to the above inequality yields
\beno
\bigl\|(\p_3v^3)_{\f12}\bigr\|_{L^3_\h\bigl(B^{\f16}_{3,2}\bigr)_\v}\lesssim \bigl\|\vss\bigr\|_{L^2}^{\f12}\bigl\|\pa_3\vss\bigr\|_{L^2}^{\f16}.
\eeno
As a result, it comes out
\begin{equation}\begin{split}\label{4.8}
\Bigl|\int_{\R^3} P_{2,\nu}(v,v) | (\p_3v^3)_{\f12}\,dx\Bigr|
&\lesssim\nu\|\pa_3\vh\|_{L_\h^6(L_\v^{\f32})}
\|\nablah v^3\|_{L_\h^2(L_\v^6)}
\bigl\|\vss\bigr\|_{L^2}^{\f12}\bigl\|\pa_3\vss\bigr\|_{L^2}^{\f16}.
\end{split}\end{equation}
While it follows from \eqref{helmdecom}
and  \eqref{2.3} that
\begin{align*}
\|\pa_3\vh\|_{L_\h^6(L_\v^{\f32})}
&\lesssim\|\pa_3\nablah\vh\|_{L^{\f32}}\\
&\lesssim\|\pa_3^2 v^3\|_{L^{\f32}}+\|\pa_3\omega\|_{L^{\f32}}\\
&\lesssim\bigl\|\vss\bigr\|_{L^2}^{\f13}
\bigl\|\pa_3 \vss\bigr\|_{L^2}
+\bigl\|\omss\bigr\|_{L^2}^{\f13}
\bigl\|\pa_3 \omss\bigr\|_{L^2},
\end{align*}
and we deduce  from Sobolev inequality
that
\beno
\begin{split}
\|\nablah v^3\|_{L_\h^2L_\v^6}\lesssim&\|v^3\|_{H^{1,\f13}}
\lesssim\|v^3\|_{H^{\f53,-\f16}}^{\f13}
\|v^3\|_{H^{\f23,\f7{12}}}^{\f23}\\
\lesssim&\|\nablah v^3\|_\hs^{\f13}
\|v^3\|_\hs^{\f16}\|\pa_3v^3\|_\hs^{\f12}. \end{split} \eeno
Substituting these above two inequalities into \eqref{4.8} leads to
\begin{equation*}\begin{split}
\Bigl|\int_{\R^3} P_{2,\nu}(v,v) | &(\p_3v^3)_{\f12}\,dx\Bigr|
\lesssim\nu^{-\f23}\Bigl(\bigl\|\vss\bigr\|_{L^2}^{\f13}
\bigl\|\nu\pa_3 \vss\bigr\|_{L^2}
+\bigl\|\omss\bigr\|_{L^2}^{\f13}
\bigl\|\nu\pa_3 \omss\bigr\|_{L^2}\Bigr)\\
&\times\|v^3\|_\hs^{\f16}\|\nablah v^3\|_\hs^{\f13}
\|\nu\pa_3v^3\|_\hs^{\f12}
\bigl\|\vss\bigr\|_{L^2}^{\f12}\bigl\|\nu\pa_3\vss\bigr\|_{L^2}^{\f16}.
\end{split}\end{equation*}
Inserting the above estimate and \eqref{4.7} into \eqref{4.5} results in
\begin{equation}\begin{split}\label{4.9}
\frac{d}{dt}\bigl\|\vss\bigr\|_{L^2}^2+
&\frac43 \bigl\|\nabla_\nu\vss\bigr\|_{L^2}^2
\lesssim\nu^{-1}\bigl\|\vss\bigr\|_{L^2}^{\f89}
\bigl\|\nabla_\nu \vss\bigr\|_{L^2}^{\f23}
\|\nabla_\nu v^3\|_{H^{\f23,-\f16}}^{\f43}\\
&+\nu^{\f13}\bigl\|\omss\bigr\|_{L^2}^{\f23}
\bigl\|\nabla_\nu \omss\bigr\|_{L^2}^2
\bigl\|\vss\bigr\|_{L^2}^{\f23}
+\nu^{-\f23}\|v^3\|_\hs^{\f16}\|\nabla_\nu v^3\|_\hs^{\f56}\\
&\qquad\qquad\quad\times\Bigl(\bigl\|\vss\bigr\|_{L^2}^{\f56}
\bigl\|\nu\pa_3 \vss\bigr\|_{L^2}^{\f76}
+\bigl\|\omss\bigr\|_{L^2}^{\f56}
\bigl\|\nu\pa_3 \omss\bigr\|_{L^2}^{\f76}\Bigr).
\end{split}\end{equation}

Let us turn to the estimate of $\nablah v^3$. Here  we
shall perform  $H^{-\f13,-\f16}$ estimate, which has the same scaling as the norm of $L^{\f32},$   for $\nablah v^3$ instead of $L^{\f32}$ estimate.
The main reason is that
 we shall  use integration by parts to handle the term,
$\D_\nu^{-1}\bigl(\partial_3 \vh\cdot\nablah v^3\bigr),$
  which appears  in the pressure~(see \eqref{4.13} below).

In fact, by taking $\hf$ inner product
of the  horizontal derivative. $\nablah,$ to the $v^3$ equation in \eqref{omegav3nu} with
$\nablah v^3$, we find
\begin{equation}\begin{split}\label{4.10}
&\f{d}{dt}\|v^3\|_{\hs}^2+2\|\nabla_\nu v^3\|_{\hs}^2
=-\sum_{i=1}^4 \bigl(Q_{i,\nu}(v,v) \, |\,\nablah v^3\bigr)_{\hf}
\quad\mbox{with}\\
&\qquad\qquad Q_{1,\nu}(v,v) \eqdefa 2\nu^3\nablah\partial_3\D_\nu^{-1}
\Bigl((\pa_3 v^3)^2+\sum_{\ell,m=1}^2\partial_\ell v^m\partial_m v^\ell\Bigr),\\
& Q_{2,\nu}(v,v)\eqdefa  4\nu^3\nablah\partial_3 \D_\nu^{-1}
\bigl(\partial_3 \vh\cdot\nablah v^3\bigr)\quad\mbox{and}\quad
Q_{3,\nu}(v,v) \eqdefa 2\nu\nablah\bigl(v\cdot \nabla v^3\bigr).
\end{split}\end{equation}

\noindent$\bullet$\underline{
The estimate of $\bigl(Q_{1,\nu}(v,v) \, |\,\nablah v^3\bigr)_{\hf}$.}

In view of \eqref{helmdecom}, we get, by applying the law of product,
 Lemma \ref{lemproductlaw}, that
\begin{equation*}\begin{split}\label{4.11}
\bigl|\bigl(Q_{1,\nu}\,|\,\nablah v^3\bigr)_{\hf}\bigr|
&\lesssim\nu^2\Bigl(\bigl\|(\pa_3v^3)^2\bigr\|_{\hf}+
\bigl\||\nablah \vh|^2\bigr\|_{\hf}\Bigr)
\|\nablah v^3\|_{\hf}\\
&\lesssim\nu^2\bigl(\|\pa_3 v^3\|_{H^{\f13,\f16}}^2+\|\omega\|_{H^{\f13,\f16}}^2\bigr)
\|v^3\|_{\hs}.
\end{split}\end{equation*}
Yet it follows from  Proposition \ref{lemomega} that
\begin{align*}
\|\pa_3 v^3\|_{H^{\f13,\f16}}^2+\|\omega\|_{H^{\f13,\f16}}^2
\lesssim &\|\pa_3 v^3\|_{H^{-\f13,\f56}}^{\f23}\|\pa_3v^3\|_{\hs}^{\frac43}
+\bigl\|\omss\bigr\|_{L^2}^{\frac23}\bigl\|\nablah\omss\bigr\|_{L^2}^{\frac43}
\bigl\|\pa_3\omss\bigr\|_{L^2}^{\frac23}\\
\lesssim &\bigl\|\vss\bigr\|_{L^2}^{\f29}\bigl\|\pa_3\vss\bigr\|_{L^2}^{\f23}
\|\pa_3 v^3\|_{\hs}^{\frac43}\\
&+\bigl\|\omss\bigr\|_{L^2}^{\frac23}\bigl\|\nablah\omss\bigr\|_{L^2}^{\frac43}
\bigl\|\pa_3\omss\bigr\|_{L^2}^{\frac23}.
\end{align*}
As a result, it comes out
\begin{equation}\begin{split}\label{4.12}
\bigl|\bigl(Q_{1,\nu}\,|\,\nablah v^3\bigr)_{\hf}\bigr|
\lesssim\Bigl(\bigl\|\vss&\bigr\|_{L^2}^{\f29}
\bigl\|\nu\pa_3\vss\bigr\|_{L^2}^{\f23}\|\nu\pa_3 v^3\|_{\hs}^{\frac43}\\
+&\nu^{\f43}\bigl\|\omss\bigr\|_{L^2}^{\frac23}\bigl\|\nablah\omss\bigr\|_{L^2}^{\frac43}
\bigl\|\nu\pa_3\omss\bigr\|_{L^2}^{\frac23}\Bigr)\|v^3\|_{\hs}.
\end{split}\end{equation}

\noindent$\bullet$\underline{
The estimate of $\bigl(Q_{2,\nu}(v,v) \, |\,\nablah v^3\bigr)_{\hf}$.}

By using integration by parts, one has
\begin{equation}\begin{split}\label{4.13}
\bigl(Q_{2,\nu}\,|\,\nablah v^3\bigr)_{\hf}=
&-4\nu^3\bigl(\nablah\pa_3\Delta_\nu^{-1}
(v^\h\cdot\na_\h v^3)\,|\,\pa_3\nablah v^3\bigl)_{\hf},\\
&-4\nu^3\bigl(\nablah\pa_3\Delta_\nu^{-1}
(v^\h\cdot\na_\h\pa_3 v^3)\,|\,\nablah v^3\bigl)_{\hf}\\
\eqdefa& \cB_{1,\nu}+\cB_{2,\nu}.
\end{split}\end{equation}
Applying the law of product, Lemma \ref{lemproductlaw}, that
\begin{align*}
|\cB_{1,\nu}|
&\lesssim \nu^2\|v^\h\cdot\na_\h v^3\|_{\hf}
\|\pa_3\nablah v^3\|_{\hf}\\
&\lesssim \nu^2\|v^\h\|_{(B^1_{2,1})_\h(H^{\f13})_{\rm v}}
\|\na_\h v^3\|_{H^{-\frac13,0}}
\|\pa_3 v^3\|_{\hs}.
\end{align*}
While thanks to \eqref{helmdecom}, we get, by applying Proposition \ref{lemomega}, that
\begin{equation}\begin{split}\label{4.14}
\|v^\h\|_{(B^1_{2,1})_\h(H^{\f13})_{\rm v}}
\lesssim&\|\nablah\vdiv\|_{H^{-\f13,\f7{12}}}^{\f23}
\|\nablah\vdiv\|_{H^{\f23,-\f16}}^{\f13}
+\|\nablah\vcurl\|_{B^{0,\f13}_{2,1}}\\
\lesssim&\bigl\|\vss\bigr\|_{L^2}^{\frac7{18}}
\bigl\|\pa_3\vss\bigr\|_{L^2}^{\frac12}
\|\pa_3v^3\|_{\hs}^{\f13}\\
&+\bigl\|\omss\bigr\|_{L^2}^{\frac12}
\bigl\|\nablah\omss\bigr\|_{L^2}^{\frac13}
\bigl\|\pa_3\omss\bigr\|_{L^2}^{\frac12}.
\end{split}\end{equation}
As a consequence, it comes out
\begin{equation}\begin{split}\label{4.15}
|\cB_{1,\nu}|
\lesssim\Bigl(&\bigl\|\vss\bigr\|_{L^2}^{\frac7{18}}
\bigl\|\nu\pa_3\vss\bigr\|_{L^2}^{\frac12}
\|\nu\pa_3v^3\|_{\hs}^{\f13}\\
&+\nu^{\f13}\bigl\|\omss\bigr\|_{L^2}^{\frac12}
\bigl\|\nablah\omss\bigr\|_{L^2}^{\frac13}
\bigl\|\nu\pa_3\omss\bigr\|_{L^2}^{\frac12}\Bigr)
\|v^3\|_{\hs}^{\f56}\|\nu\pa_3v^3\|_{\hs}^{\f76}.
\end{split}\end{equation}
Along the same line, we  deduce that
\begin{align*}
|\cB_{2,\nu}|
&\lesssim \nu^2\|v^\h\cdot\na_\h\pa_3 v^3\|_{H^{-\f13,-\f13}}
\|\nablah v^3\|_{H^{-\f13,0}}\\
&\lesssim \nu^2\|v^\h\|_{(B^1_{2,1})_\h(H^{\f13})_{\rm v}}
\|\na_\h\pa_3 v^3\|_{\hf}
\|\na_\h v^3\|_{\hf}^{\f56}\|\pa_3\na_\h v^3\|_{\hf}^{\f16},
\end{align*}
from which and \eqref{4.14},  we infer
that $|\cB_{2,\nu}|$ shares the same estimate of $|\cB_{1,\nu}|$.

Therefore, we obtain
\begin{equation}\begin{split}\label{4.16}
\bigl|\bigl(Q_{2,\nu} \,|\,\nablah v^3\bigr)_{\hf}\bigr|
&\lesssim\Bigl(\bigl\|\vss\bigr\|_{L^2}^{\frac7{18}}
\bigl\|\nu\pa_3\vss\bigr\|_{L^2}^{\frac12}
\|\nu\pa_3v^3\|_{\hs}^{\f13}\\
&+\nu^{\f13}\bigl\|\omss\bigr\|_{L^2}^{\frac12}
\bigl\|\nablah\omss\bigr\|_{L^2}^{\frac13}
\bigl\|\nu\pa_3\omss\bigr\|_{L^2}^{\frac12}\Bigr)
\|v^3\|_{\hs}^{\f56}\|\nu\pa_3v^3\|_{\hs}^{\f76}.
\end{split}\end{equation}

\noindent$\bullet$\underline{
The estimate of $\bigl(Q_{3,\nu}(v,v) \, |\,\nablah v^3\bigr)_{\hf}$.}

Due to $\dive v=0,$ by using  integration by parts, we  write
\begin{equation}\begin{split}\label{4.17}
\bigl(Q_{3,\nu} &\,|\,\nablah v^3\bigr)_{\hf}
=2\nu\bigl(\dive\nablah(v\otimes v^3)
 \,|\,\nablah v^3\bigr)_{\hf}\\
=&-2\nu\bigl(\nablah(\vh\otimes v^3) \,|\,\nablah^2 v^3\bigr)_{\hf}-2\nu\bigl(\nablah (v^3\otimes v^3) \,|\,\pa_3\nablah v^3\bigr)_{\hf}\\
\lesssim&\nu\|\vh\otimes v^3\|_\hs\|\nablah v^3\|_\hs
+\nu\bigl\|\nablah (v^3\otimes v^3)\bigr\|_\hf\|\pa_3 v^3\|_\hs.
\end{split}\end{equation}
Applying the law of product, Lemma \ref{lemproductlaw},
and  \eqref{4.14} yields
\begin{align*}
\|\vh\otimes v^3\|_\hs
\lesssim &\|v^\h\|_{(B^1_{2,1})_\h(H^{\f13})_{\rm v}}
\|v^3\|_{H^{\f23,0}}\\
\lesssim& \Bigl(\bigl\|\vss\bigr\|_{L^2}^{\frac7{18}}
\bigl\|\pa_3\vss\bigr\|_{L^2}^{\frac12}
\|\pa_3v^3\|_{\hs}^{\f13}\\
&\quad+\bigl\|\omss\bigr\|_{L^2}^{\frac12}
\bigl\|\nablah\omss\bigr\|_{L^2}^{\frac13}
\bigl\|\pa_3\omss\bigr\|_{L^2}^{\frac12}\Bigr)
\|v^3\|_\hs^{\f56}\|\pa_3 v^3\|_\hs^{\f16},
\end{align*}
and
\begin{align*}
\bigl\|\nablah (v^3\otimes v^3)\bigr\|_\hf
&\lesssim\|v^3\|_{H^{\f23,\f13}}\|\nablah v^3\|_{L^2}\\
&\lesssim\|v^3\|_\hs\|\nablah v^3\|_\hs^{\f13}
\|\pa_3 v^3\|_\hs^{\f23}.
\end{align*}
Substituting the above two estimates into \eqref{4.17} leads to
\begin{equation}\begin{split}\label{4.18}
\bigl|\bigl(Q_{3,\nu}\,|\,v^3&\bigr)_{H^{\frac12,0}}\bigr|
\lesssim\nu^{-\f23}\|v^3\|_\hs\|\nabla_\nu v^3\|_\hs^2
+\|v^3\|_\hs^{\f56}\|\nabla_\nu v^3\|_\hs^{\f76}\\
&\times\Bigl(\bigl\|\vss\bigr\|_{L^2}^{\frac7{18}}
\bigl\|\nabla_\nu\vss\bigr\|_{L^2}^{\frac12}
\|\nabla_\nu v^3\|_{\hs}^{\f13}
+\nu^{\f13}\bigl\|\omss\bigr\|_{L^2}^{\frac12}
\bigl\|\nabla_\nu\omss\bigr\|_{L^2}^{\frac56}\Bigr).
\end{split}\end{equation}

By summarizing
\eqref{4.12},~\eqref{4.16} and \eqref{4.18}, we arrive at
\begin{equation}\begin{split}\label{4.19}
\f{d}{dt}\|v^3&\|_{\hs}^2+2\|\nabla_\nu v^3\|_{\hs}^2
\lesssim \|v^3\|_\hs\Bigl(\nu^{\f43}\bigl\|\omss\bigr\|_{L^2}^{\frac23}
\bigl\|\nabla_\nu\omss\bigr\|_{L^2}^2\\
&+\bigl\|\vss\bigr\|_{L^2}^{\f29}
\bigl\|\nabla_\nu\vss\bigr\|_{L^2}^{\f23}\|\nabla_\nu v^3\|_{\hs}^{\frac43}+
\nu^{-\f23}\|\nabla_\nu v^3\|_\hs^2
\Bigr)\\
&+\|v^3\|_\hs^{\f56}\|\nabla_\nu v^3\|_\hs^{\f76}
\Bigl(\nu^{\f13}\bigl\|\omss\bigr\|_{L^2}^{\frac12}
\bigl\|\nabla_\nu\omss\bigr\|_{L^2}^{\frac56}\\
&\qquad\qquad\qquad+\bigl\|\vss\bigr\|_{L^2}^{\frac7{18}}
\bigl\|\nabla_\nu\vss\bigr\|_{L^2}^{\frac12}
\|\nabla_\nu v^3\|_{\hs}^{\f13}
\Bigr).
\end{split}\end{equation}

\subsection{Proof of Theorem \ref{thmnu}}
 We first deduce from the classical well-posedness result of
 $3$-D Navier-Stokes system
that the re-scaled system \eqref{scaleNS} admits a unique
smooth solution $v\in \cE_{T^\ast_\nu}$ for some maximal existing time $T^\ast_\nu$.

In the sequel, we shall prove that for any $t\in[0,T^\ast_\nu[,$
\begin{equation}\begin{split}\label{4.20}
&\bigl\|\omss(t)\bigr\|_{L^2}^2
+\int_0^t\bigl\|\nabla_\nu\omss(t')\bigr\|_{L^2}^2\, dt'
\leq2\, c_1^{\f32}\nu^{-1}\|\nablah v^3_0\|_{L^{\f32}}^{\f34},\\
&\bigl\|\vss(t)\bigr\|_{L^2}^2
+\int_0^t\bigl\|\nabla_\nu\vss(t')\bigr\|_{L^2}^2\, dt'
\leq 2\bigl\|\vss(0)\bigr\|_{L^2}^2=2\|\pa_3 v^3_0\|_{L^{\f32}}^{\f32},\\
&\|v^3(t)\|_{\hs}^2
+\int_0^t \|\nabla_\nu v^3(t')\|_{\hs}^2\, dt'
\leq \|v^3_0\|_{\hs}^2+\|\nablah v^3_0\|_{L^{\f32}}^2.
\end{split}\end{equation}  holds under the assumption
\eqref{smallcondition}.
In order to do so, we denote
\begin{equation}\begin{split}\label{4.21}
T^\star_\nu \eqdefa\sup\Bigl\{\, T\in ]0,T^\ast_\nu[\,:\,
\mbox{so that \eqref{4.20} holds for any $t\in[0,T]$}\, \Bigr\}.
\end{split}\end{equation}
In view of  \eqref{smallcondition}, \eqref{4.20} holds for $t\in [0,T]$ for some $T>0,$ which
guarantees that $T^\star_\nu$ is a positive number.

Next we are going to prove that $T^\star_\nu=T^\ast_\nu.$ Otherwise,
if $T^\star_\nu<T^\ast_\nu$,  for any $t<T^\star,$ by integrating
\eqref{4.4}  over $[0,t]$
we find
\beno\begin{split}
\bigl\|&\omss(t)\bigr\|_{L^2}^2
+\f43\int_0^t\bigl\|\nabla_\nu\omss(t')\bigr\|_{L^2}^2\, dt'\leq\|\om_0\|_{L^{\f32}}^{\f32}\\
&+C\nu^{-\f23}\| v^3\|_{L^\infty_t(\hs)}^{\frac12}\bigl\|\omss\bigr\|_{L^\infty_t(L^2)}^{\frac12}
\|\nabla_\nu v^3\|_{L^2_t(\hs)}^{\frac12}
\bigl\|\nabla_\nu\omss\bigr\|_{L^2_t(L^2)}^{\frac32}\\
&+C\nu^{-\f23}
\Bigl(\bigl\|\omss\bigr\|_{L^\infty_t(L^2)}^{\frac13}\bigl\|\nabla_\nu\omss\bigr\|_{L^2_t(L^2)}
+\bigl\|\vss\bigr\|_{L^\infty_t(L^2)}^{\frac13}\bigl\|\nabla_\nu\vss\bigr\|_{L^2_t(L^2)}\Bigr)\\
&\qquad\qquad\qquad\times\|v^3\|_{L^\infty_t(\hs)}^{\f13}\|\nabla_\nu v^3\|_{L^2_t(\hs)}^{\f23}
\bigl\|\omss\bigr\|_{L^\infty_t(L^2)}^{\frac13}\bigl\|\nabla_\nu\omss\bigr\|_{L^2_t(L^2)}^{\frac13},
\end{split}
\eeno
from which, and \eqref{4.20}
\footnote{By Sobolev embedding, the righthand side of the last
estimate in \eqref{4.20} can be
bounded by $C\|\nablah v^3_0\|_{L^{\f32}}^2$,
and in the following we shall also use this bound for simplicity.}, we infer
\begin{equation*}\begin{split}
\bigl\|\omss(t)\bigr\|_{L^2}^2
+\f43\int_0^t\bigl\|\nabla_\nu\omss(t')\bigr\|_{L^2}^2\, dt'&
\leq\|\om_0\|_{L^{\f32}}^{\f32}
+C\nu^{-\f23}\|\nablah v^3_0\|_{L^{\f32}}
\cdot c_1^{\f32}\nu^{-1}\|\nablah v^3_0\|_{L^{\f32}}^{\f34}\\
&+C\nu^{-\f23}\|\pa_3 v^3_0\|_{L^{\f32}}\|\nablah v^3_0\|_{L^{\f32}}
\cdot c_1^{\f12}\nu^{-\f13}\|\nablah v^3_0\|_{L^{\f32}}^{\f14}.
\end{split}\end{equation*}
Then under the assumption \eqref{smallcondition}, we deduce that
\begin{equation}\label{4.22}
\bigl\|\omss(t)\bigr\|_{L^2}^2
+\f43\int_0^t\bigl\|\nabla_\nu\omss(t')\bigr\|_{L^2}^2\, dt'
\leq(1+C\, c_2)
c_1^{\f32}\nu^{-1}\|\nablah v^3_0\|_{L^{\f32}}^{\f34}.
\end{equation}

Similarly, by integrating \eqref{4.9} over $[0,t]$
for  $t<T^\star_\nu$, we obtain
\beno
\begin{split}
\bigl\|&\vss(t)\bigr\|_{L^2}^2
+\f43\int_0^t\bigl\|\nabla_\nu\vss(t')\bigr\|_{L^2}^2\, dt'
\leq\|\pa_3 v^3_0\|_{L^{\f32}}^{\f32}\\
&+C\nu^{-1}\bigl\|\vss\bigr\|_{L^\infty_t(L^2)}^{\f89}
\bigl\|\nabla_\nu \vss\bigr\|_{L^2_t(L^2)}^{\f23}
\|\nabla_\nu v^3\|_{L^2_t(H^{\f23,-\f16})}^{\f43}\\
&+C\nu^{\f13}\bigl\|\omss\bigr\|_{L^\infty_t(L^2)}^{\f23}
\bigl\|\nabla_\nu \omss\bigr\|_{L^2_t(L^2)}^2
\bigl\|\vss\bigr\|_{L^\infty_t(L^2)}^{\f23}\\
&+C\nu^{-\f23}\|v^3\|_{L^\infty_t(\hs)}^{\f16}\|\nabla_\nu v^3\|_{L^2_t(\hs)}^{\f56}
\Bigl(\bigl\|\omss\bigr\|_{L^\infty_t(L^2)}^{\f56}
\bigl\|\nu\pa_3 \omss\bigr\|_{L^2_t(L^2)}^{\f76}\\
&\qquad\qquad+\bigl\|\vss\bigr\|_{L^\infty_t(L^2)}^{\f56}
\bigl\|\nu\pa_3 \vss\bigr\|_{L^2_t(L^2)}^{\f76}
\Bigr),
\end{split}
\eeno
which together with \eqref{4.20} ensures that
\begin{equation*}\begin{split}
\bigl\|&\vss(t)\bigr\|_{L^2}^2
+\f43\int_0^t\bigl\|\nabla_\nu\vss(t')\bigr\|_{L^2}^2\, dt'
\leq\|\pa_3 v^3_0\|_{L^{\f32}}^{\f32}\\
&+C\nu^{-1}\|\pa_3 v^3_0\|_{L^{\f32}}^{\f76}
\|\nablah v^3_0\|_{L^{\f32}}^{\f43}+C\nu^{\f13}\|\pa_3 v^3_0\|_{L^{\f32}}\cdot
c_1^{\f32}\nu^{-1}\|\nablah v^3_0\|_{L^{\f32}}^{\f34}\\
&+C\nu^{-\f23}\|\nablah v^3_0\|_{L^{\f32}}
\Bigl(\|\pa_3 v^3_0\|_{L^{\f32}}^{\f32}
+c_1^{\f32}\nu^{-1}\|\nablah v^3_0\|_{L^{\f32}}^{\f34}\Bigr).
\end{split}\end{equation*}
It is easy to observe from \eqref{smallcondition} that
\begin{align*}
C\nu^{-1}\|\pa_3 v^3_0\|_{L^{\f32}}^{\f76}
\|\nablah v^3_0\|_{L^{\f32}}^{\f43}
&=C\nu^{-1}\|\nablah v^3_0\|_{L^{\f32}}\cdot
\|\pa_3 v^3_0\|_{L^{\f32}}^{\f23}
\|\nablah v^3_0\|_{L^{\f32}}^{\f13}\cdot
\|\pa_3 v^3_0\|_{L^{\f32}}^{\f12}\\
&\leq C\,c_2^{\f23}
\|\pa_3 v^3_0\|_{L^{\f32}}^{\f32},
\end{align*}
\begin{align*}
C\nu^{\f13}\|\pa_3 v^3_0\|_{L^{\f32}}\cdot
c_1^{\f32}\nu^{-1}\|\nablah v^3_0\|_{L^{\f32}}^{\f34}
&=C\,c_1^{\f32}\nu^{-\f12}\|\nablah v^3_0\|_{L^{\f32}}^{\f12}
\cdot\nu^{-\f16}\|\nablah v^3_0\|_{L^{\f32}}^{\f14}
\cdot\|\pa_3 v^3_0\|_{L^{\f32}}\\
&\leq C\,c_1^{\f76}c_2^{\f14}\|\pa_3 v^3_0\|_{L^{\f32}}^{\f32},
\end{align*}
and
\begin{align*}
C\nu^{-\f23}\|\nablah v^3_0\|_{L^{\f32}}
\Bigl(\|\pa_3 v^3_0\|_{L^{\f32}}^{\f32}
+c_1^{\f32}\nu^{-1}\|\nablah v^3_0\|_{L^{\f32}}^{\f34}\Bigr)
\leq C\bigl(c_2+c_1^{\f12}c_2^{\f14}\bigr)
\|\pa_3 v^3_0\|_{L^{\f32}}^{\f32}.
\end{align*}
Hence as long as  $c_1,~c_2$ in \eqref{smallcondition} are sufficiently small, we achieve
\begin{equation}\label{4.23}
\bigl\|\vss(t)\bigr\|_{L^2}^2
+\f43\int_0^t\bigl\|\nabla_\nu\vss(t')\bigr\|_{L^2}^2\, dt'
\leq\bigl(1+C\,c_2^{\f14}\bigr)
\|\pa_3 v^3_0\|_{L^{\f32}}^{\f32}
\end{equation}

Exactly along the same line, we deduce from \eqref{4.19} that
for any $t\in[0,T^\star_\nu[$,
\begin{equation*}\begin{split}
&\|v^3(t)\|_{\hs}^2
+2\int_0^t \|\nabla_\nu v^3(t')\|_{\hs}^2\, dt'
\leq\|v^3_0\|_{\hs}^2\\
&+C\|\nablah v^3_0\|_{L^{\f32}}
\Bigl(\|\pa_3 v^3_0\|_{L^{\f32}}^{\f23}
\|\nablah v^3_0\|_{L^{\f32}}^{\frac43}+c_1^2\|\nablah v^3_0\|_{L^{\f32}}
+\nu^{-\f23}\|\nablah v^3_0\|_{L^{\f32}}^2\Bigr)\\
&+C\|\nablah v^3_0\|_{L^{\f32}}^2
\Bigl(\|\pa_3 v^3_0\|_{L^{\f32}}^{\frac23}
\|\nablah v^3_0\|_{L^{\f32}}^{\f13}
+c_1\nu^{-\f13}\|\nablah v^3_0\|_{L^{\f32}}^{\f12}\Bigr).
\end{split}\end{equation*}
Then under the assumption \eqref{smallcondition}, we obtain
\begin{equation}\label{4.24}
\|v^3(t)\|_{\hs}^2
+2\int_0^t \|\nabla_\nu v^3(t')\|_{\hs}^2\, dt'
\leq\|v^3_0\|_{\hs}^2+C\,(c_1^2+c_2)\|\nablah v^3_0\|_{L^{\f32}}^2.
\end{equation}

Hence if the constants $c_1$ and $c_2$  in \eqref{smallcondition} are small enough,
 we  deduce from  \eqref{4.22}-\eqref{4.24} that
for any $t\in[0,T^\star_\nu[$,
\begin{align*}
&\ \bigl\|\omss(t)\bigr\|_{L^2}^2
+\int_0^t\bigl\|\nabla_\nu\omss(t')\bigr\|_{L^2}^2\, dt'
\leq\f32\,c_1^{\f32}\nu^{-1}\|\nablah v^3_0\|_{L^{\f32}}^{\f34},\\
&\bigl\|\vss(t)\bigr\|_{L^2}^2
+\int_0^t\bigl\|\nabla_\nu\vss(t')\bigr\|_{L^2}^2\, dt'
\leq \f32\|\pa_3 v^3_0\|_{L^{\f32}}^{\f32},\\
\|v^3(t&)\|_{\hs}^2
+\int_0^t \|\nabla_\nu v^3(t')\|_{\hs}^2\, dt'
\leq \|v^3_0\|_{\hs}^2+\f12\|\nablah v^3_0\|_{L^{\f32}}^2,
\end{align*}
which contracts with the definition of $T^\star_\nu$
given by \eqref{4.21}. This in turn shows that $T^\star_\nu=T^\ast_\nu$,
and \eqref{4.20} holds for any $t\in[0,T^\ast_\nu[$.

Now if $T^\ast_\nu<\infty$, for any $p\in]4,6[,$ we have
\begin{align*}
\int_0^{T^\ast_\nu}\|v^3(t)\|_{H^{\frac12+\frac2p}}^p\,dt
&\lesssim \int_0^{T^\ast_\nu}\|v^3(t)\|_{H^{\frac12+\frac2p,0}}^p\,dt
+\int_0^{T^\ast_\nu}\|\pa_3v^3(t)\|_{H^{0,\f2p-\f12}}^p\,dt\\
&\leq C\sup_{t\in [0,T^\ast_\nu[}\|v^3(t)\|_{\hs}^{p-2}
\int_0^{T^\ast_\nu}\|\nablah v^3(t)\|_{\hs}^{2-\f p6}
\|\pa_3 v^3(t)\|_{\hs}^{\f p6}\,dt\\
&\qquad\qquad+C\sup_{t\in [0,T^\ast_\nu[}
\bigl\|\vss\bigr\|_{L^2}^{\bigl(\f{4}3-\f2p\bigr)p}
\int_0^{T^\ast_\nu}\bigl\|\nabla \vss(t)\bigr\|_{L^2}^2\,dt\\
&\leq C\|\nabla v^3_0\|_{L^{\f32}}^p,
\end{align*}
which contradicts with the blow-up criterion \eqref{blowupCZ5}.
As a result, $T^\ast_\nu=\infty$, and
\eqref{4.20} holds for any $t\in[0,\infty[$.
This completes the proof of Theorem \ref{thmnu} under the assumption \eqref{smallcondition}.

Theorem \ref{thmnu} with the condition \eqref{smallcondition2} can be proved along the same line.
One just need to modify the estimate \eqref{4.20} to be
\begin{equation*}\begin{split}
&\bigl\|\omss(t)\bigr\|_{L^2}^2
+\int_0^t\bigl\|\nabla_\nu\omss(t')\bigr\|_{L^2}^2\, dt'
\leq2\, c_1^{\f32}\nu^{-1}\|\nablah v^3_0\|_{L^{\f32}}^{\f34},\\
&\bigl\|\vss(t)\bigr\|_{L^2}^2
+\int_0^t\bigl\|\nabla_\nu\vss(t')\bigr\|_{L^2}^2\, dt'
\leq 2\|\nabla v^3_0\|_{L^{\f32}}^{\f32},\\
&\|v^3(t)\|_{\hs}^2
+\int_0^t \|\nabla_\nu v^3(t')\|_{\hs}^2\, dt'
\leq \|v^3_0\|_{\hs}^2+\|\nablah v^3_0\|_{L^{\f32}}^2.
\end{split}\end{equation*}
And then by repeating the previous argument, we  can conclude the proof of Theorem \ref{thmnu}, which we shall not present details here.

\bigbreak \noindent {\bf Acknowledgments.} P. Zhang would like to thank Professor L.
 Caffarelli and Professor J.-Y. Chemin for profitable discussions on this topic.

P. Zhang is partially supported
by NSF of China under Grants   11371347 and 11688101, Morningside Center of Mathematics of The Chinese Academy of Sciences and innovation grant from National Center for
Mathematics and Interdisciplinary Sciences.

\end{document}